\pdfoutput=1
\def\arXiv{1} 
\documentclass[11pt]{article}

\newcommand{\notarxiv}[1]{foo}
\newcommand{\arxiv}[1]{ba}
\ifdefined \arXiv
\renewcommand{\arxiv}[1]{#1}%
\renewcommand{\notarxiv}[1]{\ignorespaces}%
\else%
\renewcommand{\arxiv}[1]{\ignorespaces}%
\renewcommand{\notarxiv}[1]{#1}%
\fi%

\usepackage{thmtools}
\usepackage{thm-restate}

 \notarxiv{
 	\undef\corollary
 	\undef\definition
	\setcitestyle{numbers,square,comma} 
	\declaretheorem[name=Lemma,sibling=theorem]{lem}
	\declaretheorem[name=Proposition,sibling=theorem]{prop}
	\declaretheorem[name=Corollary,sibling=theorem]{corollary}
	\declaretheorem[name=Definition,sibling=theorem]{definition}

	\undef\lemma
	\undef\proposition	
}

\arxiv{
	\usepackage{amsthm}
	\usepackage[linesnumbered,ruled,vlined,boxed,algo2e]{algorithm2e}

	\usepackage[dvipsnames]{xcolor}
	\usepackage[top=1in, right=1in, left=1in, bottom=1in]{geometry}
	\usepackage[numbers,square]{natbib}
	\usepackage{url}
	\usepackage[hidelinks]{hyperref}
	\hypersetup{
		colorlinks=true,
		linkcolor=blue!70!black,
		citecolor=blue!70!black,
		urlcolor=blue!70!black}
	
	\theoremstyle{plain}
	
	\newtheorem{theorem}{Theorem}
	\newtheorem{lem}{Lemma}

	\newtheorem{corollary}{Corollary}
	\newtheorem{definition}{Definition}
	
	\theoremstyle{definition}

	\newtheorem*{example*}{Example}

}

\usepackage{enumitem}
\usepackage{amssymb}
\usepackage{amsbsy}
\usepackage{amsmath}

\usepackage{cleveref}
\Crefname{lem}{Lemma}{Lemmas}
\Crefname{prop}{Proposition}{Propositions}

\usepackage{amsfonts}
\usepackage{latexsym}
\usepackage{graphicx}
\usepackage{color}
\usepackage{xifthen}
\usepackage{xspace}
\usepackage{mathtools}
\usepackage{bbm}
\usepackage{multirow}
\usepackage[normalem]{ulem}
\usepackage{array}
\usepackage[utf8]{inputenc}
\usepackage[T1]{fontenc}
\usepackage{etoolbox}
\newtoggle{restatements}

\newtoggle{heavyplots}

\usepackage{xfrac}

\usepackage{titlesec}

\usepackage{setspace}

\usepackage{esvect}

\usepackage{nicefrac}

\usepackage{cases}
\usepackage{empheq}

 \usepackage{booktabs}
 \usepackage{multirow}
 
 \usepackage{caption}

\DeclarePairedDelimiter{\abs}{\lvert}{\rvert} %
\DeclarePairedDelimiter{\brk}{[}{]}
\DeclarePairedDelimiter{\crl}{\{}{\}}
\DeclarePairedDelimiter{\prn}{(}{)}

\DeclarePairedDelimiter{\norm}{\|}{\|}

\DeclarePairedDelimiter{\ceil}{\lceil}{\rceil}
\DeclarePairedDelimiter{\floor}{\lfloor}{\rfloor}

\newcommand{\overeq}[1]{\overset{#1}{=}}
\newcommand{\overle}[1]{\overset{#1}{\le}}
\newcommand{\overge}[1]{\overset{#1}{\ge}}

\NewDocumentCommand\Ex{s O{} m }{%
	\mathbb{E}%
	\begingroup
	\IfBooleanTF{#1}
	{\ExInn*{#3}}
	{\ExInn[#2]{#3}}%
	\endgroup
}

\DeclarePairedDelimiterX\ExInn[1]{[}{]}{%
	\activatebar
	#1%
}

\RenewDocumentCommand\Pr{sO{}r()}{%
	\mathbb{P}%
	\begingroup
	\IfBooleanTF{#1}
	{\PrInn*{#3}}
	{\PrInn[#2]{#3}}%
	\endgroup
}

\DeclarePairedDelimiterX\PrInn[1](){%
	\activatebar
	#1%
}

\newcommand{\activatebar}{%
	\begingroup\lccode`~=`|
	\lowercase{\endgroup\def~}{\;\delimsize\vert\;}%
	\mathcode`|=\string"8000
}

\newcommand\numberthis{\addtocounter{equation}{1}\tag{\theequation}}

\newcommand{\normBig}[1]{\Big\|{#1}\Big\|} %
\newcommand{\R}{\mathbb{R}} %
\newcommand{\N}{\mathbb{N}} %
\newcommand{\E}{\mathbb{E}} %
\renewcommand{\P}{\mathbb{P}}	%
\newcommand{\I}{\mathbb{I}} %

\usepackage{xcolor}
\usepackage[algo2e]{algorithm2e}

\SetCommentSty{mycommfont}
\SetKwInput{KwInput}{Input} 
\SetKwInput{KwParameter}{Parameters}                %
\SetKwInput{KwOutput}{Output}              %
\SetKwInput{KwReturn}{Return}              %
\SetKwComment{Comment}{$\triangleright$\ }{}
\SetCommentSty{mycommfont}

\makeatletter
\let\oldnl\nl%
\newcommand{\nonl}{\renewcommand{\nl}{\let\nl\oldnl}}%
\makeatother

\DeclareMathOperator*{\argmin}{arg\,min}

\providecommand{\abs}{\mathop{\rm abs}}

\providecommand{\minimize}{\mathop{\rm minimize}}

\newcommand{\half}{\frac{1}{2}}

\newcommand{\defeq}{\coloneqq}

\newcommand{\grad}{\nabla}
\newcommand{\hess}{\nabla^2}

\newcommand{\eps}{\epsilon}

\newcommand{\Otil}[1]{\widetilde{O}( #1 )}
\newcommand{\Omtil}[1]{\widetilde{\Omega}( #1 )}
\newcommand{\Otilb}[1]{\widetilde{O}\left( #1 \right)}

\newcommand{\inner}[2]{\left<#1,#2\right>}

\newcommand{\ball}{\mathbb{B}}

\newcommand{\indic}[1]{\I{\{#1\}}}

\newcommand{\veps}{\varepsilon}

\usepackage[textsize=scriptsize,textwidth=2cm]{todonotes}
\newcommand{\yairside}[1]{\todo[color=blue!10]{Yair: #1}}
\newcommand{\yair}[1]{{\bf \color{blue} Yair: #1}}

\newcommand{\sidford}[1]{{\bf \color{purple} Sidford: #1}}
\newcommand{\arun}[1]{{\bf \color{orange} Arun: #1}}

  \renewcommand{\yairside}[1]{\ignorespaces}
  \renewcommand{\yair}[1]{\ignorespaces}
  \renewcommand{\sidford}[1]{\ignorespaces}
  \renewcommand{\arun}[1]{\ignorespaces}

\newcommand{\opt}{_\star}
\newcommand{\xopt}{x\opt}

\newcommand{\F}{F_{\max}}
\newcommand{\Fsm}[1][\epsilon]{F_{\mathrm{smax},{#1}}}

\newcommand{\beF}{\Gamma_{\epsilon,\lambda}}

\newcommand{\bx}{\bar{x}}  %
\newcommand{\reps}{r_\epsilon}

\newcommand{\Lf}{L_f}  %
\newcommand{\Lg}{L_g}  %
\newcommand{\lip}{\Lf} %

\newcommand{\domain}{\R^d}

\newcommand{\bprox}[2][\lambda,r]{\mathrm{bprox}_{#1}^{#2}}
\newcommand{\prox}[2][\lambda]{\mathrm{prox}_{#1}^{#2}}

\newcommand{\coind}[1]{_{[#1]}}  %

\newcommand{\gradient}{\nabla}
\newcommand{\hessian}{\gradient^2}
\newcommand{\identityMatrix}{I}

\newcommand{\oracle}[2][\lambda,\delta]{\mathcal{O}_{#1}(#2)}
\newcommand{\oracles}[1][\lambda,\delta]{\mathcal{O}_{#1}}

\newcommand{\Tf}{\mathcal{T}_f}
\newcommand{\Tg}{\mathcal{T}_g}

\newcommand{\poly}{\mathsf{poly}}
\newcommand{\Thetatil}[1]{\widetilde{\Theta}(#1)}

\newcommand{\xhat}{\hat{x}}
\newcommand{\hx}{\xhat}
\newcommand{\linesearch}{\lambda\textsc{-Bisection}}

\newcommand{\q}{\tilde{y}}

\newcommand{\xtilde}{\widetilde{x}}

\newcommand{\BOO}{BROO\xspace}

\notarxiv{
	\usepackage{times}
}

\notarxiv{
	\title[Thinking Inside the Ball]{Thinking Inside the Ball:\\ 
	Near-Optimal Minimization of the Maximal Loss}

\coltauthor{%
	\Name{Yair Carmon} \Email{ycarmon@cs.tau.ac.il}\\
	\addr Tel Aviv University
	\AND
	\Name{Arun Jambulapati} \Email{jmblpati@stanford.edu}\\
	\addr Stanford University%
	\AND
	\Name{Yujia Jin} \Email{yujiajin@stanford.edu}\\
	\addr Stanford University%
	\AND
	\Name{Aaron Sidford} \Email{sidford@stanford.edu}\\
	\addr Stanford University%
}
}
\arxiv{
	\title{Thinking Inside the Ball:\\ 
	Near-Optimal Minimization of the Maximal Loss}

	\author{Yair Carmon ~~~ Arun Jambulapati ~~~ Yujia Jin ~~~ Aaron Sidford\\
	\href{mailto:ycarmon@cs.tau.ac.il}{\texttt{ycarmon@cs.tau.ac.il}}, 
	\texttt{\{\href{mailto:jmblpati@stanford.edu}{jmblpati},%
		\href{mailto:yujiajin@stanford.edu}{yujiajin},%
		\href{mailto:sidford@stanford.edu}{sidford}\}@stanford.edu}}
	\date{}
}

\begin{document}

\maketitle

\begin{abstract}%
\label{abstract}
We characterize the complexity of minimizing $\max_{i\in[N]} f_i(x)$ for convex, Lipschitz functions $f_1,\ldots, f_N$. 
For non-smooth functions, existing methods require $O(N\epsilon^{-2})$ queries to a first-order oracle  to compute an $\epsilon$-suboptimal point and $\Otil{N\epsilon^{-1}}$ queries if the $f_i$ are  $O(1/\epsilon)$-smooth.
We develop methods with improved complexity bounds of $\Otil{N\epsilon^{-2/3} + \epsilon^{-8/3}}$ in the non-smooth case and $\Otil{N\epsilon^{-2/3} + \sqrt{N}\epsilon^{-1}}$ in the $O(1/\epsilon)$-smooth case.
Our methods consist of a recently proposed ball optimization oracle acceleration algorithm (which we refine) and a careful implementation of said oracle for the softmax function. %
We also prove an oracle complexity lower bound scaling as $\Omega(N\epsilon^{-2/3})$, showing that our dependence on $N$ is optimal up to polylogarithmic factors. 
\end{abstract}

\section{Introduction}

Consider the problem of approximately minimizing the maximum of $N$ convex functions: given $f_1, \ldots, f_N$ such that for every $i\in [N]$ the function $f_i: \R^d \to \R$ is convex, Lipschitz and possibly smooth, and a target accuracy $\epsilon$, 
\begin{equation}\label{eq:problem}
\mbox{find a point $x$ such that }~
\F(x) - \inf_{\xopt \in \R^d} \F(\xopt) \le \epsilon
	~~\mbox{where}~~
	\F(x) \defeq \max_{i\in[N]} f_i(x) ~.
\end{equation}

Problems of this form play significant roles in optimization and machine learning. The maximum of $N$ functions is a canonical example of structured non-smoothness and several works develop methods for exploiting it~\cite{nesterov2005smooth,nemirovski2009robust,shalev2016minimizing,bullins2020highly,carmon2020acceleration}. The special case where the $f_i$'s are linear functions is particularly important for machine learning, since it is equivalent to hard-margin SVM training (with $f_i$ representing the negative margin on the $i$th example)~\cite{vapnik1999overview,clarkson2012sublinear,hazan2011beating}. 
Going beyond the linear case, \citet{shalev2016minimizing} argue that minimizing the maximum classification loss can have advantageous effects on training speed and generalization in the presence of rare informative examples. 
Moreover, minimizing the worst-case objective is the basic paradigm of robust optimization~\cite{bental2009robust,namkoong2016stochastic}. In particular, since $\F(x)=\max_{p\in \Delta^N} \sum_{i\in[N]} p_i f_i(x)$ the problem corresponds to an extreme case of distributionally robust optimization~\cite{bental2013robust} with an uncertainty set that encompasses the entire probability simplex $\Delta^N$.

The goal of this paper is to characterize the complexity of this fundamental problem. We are particularly interested in the regime where the number of data points $N$ and the problem dimension $d$ are large compared to the desired level of accuracy $1/\epsilon$, as is common in modern machine learning.
Consequently, we focus on dimension-independent first-order methods (i.e., methods which only rely on access to $f_i(x)$ and a (sub)gradient $\grad f_i(x)$
as opposed to higher-order derivatives), and report complexity in terms of the number of function/gradient evaluations required to solve the problem.

\subsection{Related work}
To put our new complexity bounds in context, we first review the prior art in solving the problem~\eqref{eq:problem} with first-order methods. For simplicity of presentation, throughout the introduction we assume  each $f_i$ is 1-Lipschitz and that $\F$ has a global minimizer $\xopt$ with (Euclidean) norm at most 1. 
The simplest approach to solving the problem~\eqref{eq:problem} is the subgradient method~\cite{nesterov2018lectures}. This method finds an $\epsilon$-accurate solution in $O(\epsilon^{-2})$ iterations, with each step computing a subgradient of $\F$, which in turn requires evaluation of all $N$ function values and a single gradient. Consequently, the complexity of this method is $O(N\epsilon^{-2})$. We are unaware of prior work obtaining improved complexity without further assumptions.\footnote{The center of gravity method \cite{levin1965cg,newman65cg} yields a query complexity $O(N d \log(1/\epsilon))$ which is an improvement only for sufficiently small problem dimension $d$.}

However, even a weak bound on smoothness helps: if each $f_i$ has $O(1/\epsilon)$-Lipschitz gradient, then it is possible to minimize $\F$ to accuracy $\epsilon$ with complexity $\Otil{N\epsilon^{-1}}$ ~\cite{nesterov2005smooth}.\footnote{Throughout the paper, the  $\Otil{\cdot}$ and $\Omtil{\cdot}$ hide polylogarithmic factors. } This result relies on the so-called ``softmax'' approximation of the maximum, 
\begin{equation}\label{eq:softmax}
	\Fsm(x) \defeq 
	\epsilon' \log\prn*{\sum_{i\in[N]} e^{f_i(x) / \epsilon'}},~~\mbox{where }
	\epsilon' = \frac{\epsilon}{2\log N}.
\end{equation}
It is straightforward to show that $\abs{\Fsm(x)-\F(x)} \le \frac{\epsilon}{2}$
for all $x\in\R^d$, 
and that  $\grad \Fsm$ is $\Otil{1/\epsilon}$-Lipschitz if $\grad f_i$ is $O(1/\epsilon)$-Lipschitz for every $i$. Therefore, Nesterov's accelerated gradient descent~\cite{nesterov2005smooth} finds a minimizer of $\Fsm$ to accuracy $\frac{\epsilon}{2}$ in $\Otil{\sqrt{1/\epsilon}/\sqrt{\epsilon}}$ iterations, with each iteration requiring $N$ evaluations of $f_i$ and $\grad f_i$ to compute $\grad \Fsm$, yielding the  claimed bound.
The assumption that $\grad f_i$ is $O(1/\epsilon)$-Lipschitz is fairly weak; see \Cref{sec:app-discussion-smoothness} for additional discussion.

 Given more smoothness, further improvement is possible. \citet[Section 2.3.1]{nesterov2018lectures} shows that it suffices to solve $O(\sqrt{\Lg/\epsilon})$ linearized subproblems of the form $\min_{x\in\R^d} \max_{i\in[N]}\big\{ f_i(y_t) +(\grad f_i(y_t))^\top (x-y_t) + \frac{\Lg}{2}\norm{x-y_t}^2 \big\}$. This yields a query complexity upper bound of $O(N\sqrt{\Lg/\epsilon})$, Though the complexity of solving each subproblem is not immediately clear, in  \Cref{sec:app-discussion-nesterov} we explain how a first-order method \cite{carmon2019variance} solves the subproblem to sufficient precision. Additional schemes for solving~\eqref{eq:problem} in the special case of linear functions (i.e., $\Lg=0$) are discussed in \Cref{sec:app-discussion-linear}.

A powerful technique for solving optimization problems with a large number $N$ of component functions is sampling components in order to compute cheap  unbiased gradient estimates. However, both $\F$ and $\Fsm$ are not given as linear combinations of the $f_i$'s. Consequently, it is not clear how to efficiently compute unbiased estimators for their gradients.
Several works address this by considering the saddle point problem
\begin{equation*}
	\min_{x\in\R^d}\max_{p\in\Delta^N} F_{\mathrm{pd}}(x;p) \defeq \sum_{i\in [N]} p_i f_i(x),
\end{equation*}
which is equivalent to minimizing to $\F$. One can  obtain unbiased  estimators for $\grad F_{\mathrm{pd}}(x;p)$, and apply stochastic mirror descent to find its saddle-point~\cite{nemirovski2009robust,shalev2016minimizing,namkoong2016stochastic}. However, all known estimators for $\grad_p F_{\mathrm{pd}}$ have complexity-variance product $\Omega(N)$. Consequently, the best general guarantees known for such methods are $\Otil{N\epsilon^{-2}}$ iterations and total  complexity.\footnote{
Exact-gradient primal-dual methods such as mirror-prox~\cite{nemirovski2004prox} and dual-extrapolation~\cite{nesterov2007dual} have complexity guarantees scaling as $\Otil{N\epsilon^{-1}}$ under the stronger smoothness assumption $\Lg = O(1)$ \cite[cf.][Section 5.2.4]{bubeck2015convex}. 
}
 \citet{shalev2016minimizing} analyze a stochastic primal-dual method from an online learning perspective. They show that if the online method producing the primal updates admits a mistake bound (as is the case for learning halfspaces), then the complexity of the approach improves to $\Otil{N\epsilon^{-1}}$.  We show that adopting a primal-only perspective and iteratively restricting $x$ to a small ball (i.e., ``thinking inside the ball'') allows us to make better use of the scalability of stochastic gradient methods.

 \begin{table}[]
 	\captionsetup{font=small}
 	\centering
 	\begin{tabular}{@{}llll@{}}
 		\toprule
 		Smoothness                                       & Method                            & Upper bound                                 & Lower bound                                                 \\ \midrule
 		\multirow{2}{*}{None ($\Lg=\infty$)}             & Subgradient method                & $N\epsilon^{-2}$                           & \multirow{2}{*}{$N\epsilon^{-2/3} + \epsilon^{-2}$}         \\
 		& Ours                              & $N\epsilon^{-2/3} + \epsilon^{-8/3}$       &                                                             \\ \midrule
 		\multirow{2}{*}{Weak ($\Lg \approx 1/\epsilon$)} & AGD on softmax                    & $N\epsilon^{-1}$                           & \multirow{2}{*}{$N\epsilon^{-2/3} + \sqrt{N}\epsilon^{-1}$} \\
 		& Ours                              & $N\epsilon^{-2/3} + \sqrt{N}\epsilon^{-1}$ &                                                             \\ \midrule
 		Strong ($\Lg \ll 1/\epsilon$)                    & AGD on linearization* & $N\sqrt{\Lg\epsilon^{-1}}$       & $N\Lg^{1/3}\epsilon^{-1/3} + \sqrt{N\Lg\epsilon^{-1}}$      \\ \bottomrule
 	\end{tabular}
 	\caption{\label{table:summary} The complexity of solving the problem~\eqref{eq:problem} in terms of number of $(i,x)$ queries for computing and $f_i(x)$ and $\grad f_i(x)$. The tables assume each $f_i$ is convex, 1-Lipschitz and (optionally) has $\Lg$-Lipschitz gradient, and that $\F$ has a minimizer with norm at most 1. The stated rates omit constant and (in the upper bounds) polylogarithmic factors. *For this algorithm only, the computational complexity is not simply $d$ times the query complexity; see~\Cref{sec:app-discussion-nesterov}.}
 \end{table}

\subsection{Our contributions}
To motivate our developments, %
note that the general complexity guarantees described above all scale linearly with the number of functions $N$. 
On the one hand, this is to be expected, as even evaluating the maximum of $N$ numbers requires querying all of them. On the other hand, a linear scaling in $N$ stands in sharp contrast to guarantees for minimizing the \emph{average} of $N$ functions, which are typically sublinear in $N$. Since good scaling with dataset size is crucial in machine learning, we wish to precisely characterize the number of dataset passes (that is, the coefficient of $N$) in the complexity of minimizing $\F$. 
 
 Towards that end, we prove an oracle complexity lower bound. The bound shows that any algorithm that operates by repeatedly querying $i,x$ and observing $f_i(x), \grad f_i(x)$, must make $\Omega(N\epsilon^{-2/3})$ queries in order to solve problem~\eqref{eq:problem} for some convex, 1-Lipschitz problem instance $f_1,\ldots, f_N$ with domain in the unit ball. The same bound continues to hold even when constraining the $f_i$ to have $O(1/\epsilon)$-Lipschitz gradient, and when using high-order derivative oracles. This result further sharpens the contrast to average risk minimization, as it implies $\Omega(\epsilon^{-2/3})$ dataset passes are required in the worst case. However, it also suggests the potential for significant improvement over existing algorithms and their complexity bounds.

 We realize this potential with new algorithms whose leading complexity term in $N$ matches our lower bound up to polylogarithmic factors. In the non-smooth case, our approach solves~\eqref{eq:problem} with complexity $\Otil{N\epsilon^{-2/3} + \epsilon^{-8/3}}$, dominating prior guarantees for $N=\Omtil{\epsilon^{-2/3}}$. For $O(1/\epsilon)$-Lipschitz gradient functions, we obtain the stronger rate $\Otil{N\epsilon^{-2/3} + \sqrt{N}\epsilon^{-1}}$, which dominates prior guarantees for $N=\Omtil{1}$. 
 At the core of these algorithms is a technique for accelerated optimization given a ball optimization oracle~\cite{carmon2020acceleration}; we make several improvements to this technique, which may be of independent interest.

 \Cref{table:summary} summarizes our results and their comparison to prior art. In addition to the results described above, the table also contains lower bounds on sublinear terms in $N$ (that follow from standard arguments), as well as a lower bound for the smooth regime where $\Lg=o(1/\epsilon)$. In this regime there exists a gap between the linear terms in the upper and lower bounds.

 \subsection{Overview of techniques}

Our algorithms rely on a new technique introduced by~\citet{carmon2020acceleration} for acceleration with a ball optimization oracle (BOO). For any $r>0$ and $F:\R^d \to \R$, a BOO of radius $r$ takes  in a query point $\bx\in\R^d$ and returns an (approximate) minimizer of $F$ in a ball of radius $r$ around $\bx$. The technique, which is a variant of  Monteiro-Svaiter acceleration~\cite{monteiro2013accelerated,gasnikov19near,bubeck2019complexity,bullins2020highly},  minimizes $F$ to $\epsilon$ accuracy using $\Otil{(1/r)^{2/3}}$ oracle calls (with $\poly(\log(1/\epsilon))$ factors hidden). \citet{carmon2020lower}  apply their technique to the special case of~\eqref{eq:problem} with linear losses (see also~\Cref{sec:app-discussion-linear}), showing that the log-sum-exp function is quasi-self-concordant and implementing a BOO of radius $r=\Thetatil{\epsilon}$ using $\Otil{1}$ linear system solves. However, this approach does not extend to general $f_i$ because quasi-self-concordance no longer holds for $\Fsm$, which might not even be differentiable.

The main technical insight of our paper is that it is possible to efficiently implement a BOO of radius $\reps=\Thetatil{\epsilon}$ for $\Fsm$ using stochastic first-order methods. More precisely, for any $\bx \in\R^d$ we can minimize  $\Fsm$ in a ball of radius $\reps$ around $\bx$ to any $\poly(\epsilon)$ accuracy with precisely $N$ function evaluations and $\poly(1/\epsilon)$ (sub-)gradient evaluations. Using BOO acceleration, this immediately implies an $\Otil{N\epsilon^{-2/3} + \poly(1/\epsilon)}$ complexity bound exhibiting optimal $N$ dependence. 

To implement the BOO for $\Fsm$, we consider instead the ``exponentiated softmax'' function
\begin{equation*}
	\Gamma_{\epsilon}(x) = {\epsilon'} \cdot \exp\prn*{\frac{\Fsm(x)-\Fsm(\bx)}{{\epsilon'}}} = \sum_{i\in [N]} p_i {\epsilon'}  \cdot e^{ \frac{f_i(x)-f_i(\bx)}{{\epsilon'}}}
	~~\mbox{where}~~p_i = \frac{e^{f_i(\bx)/{\epsilon'}}}{\sum_{j\in[N]} e^{f_i(\bx)/{\epsilon'}}},
\end{equation*}
and $\epsilon'=\epsilon/(2\log N)$ as in~\cref{eq:softmax}.
Note that $\Gamma_\epsilon$ is a monotonically increasing transformation of $\Fsm$, and is therefore convex with the same minimizer as $\Fsm$. Moreover, it is a (weighted) finite sum,
and consequently amenable to stochastic gradient methods. It remains to verify that the functions $\xi_i(x) = \epsilon' \cdot e^{ ({f_i(x)-f_i(\bx)})/{\epsilon'}}$ are well-behaved, which might look difficult since exponentials are notoriously unstable. However, our choice of $r$ and Lipschitz continuity of $f_i$ implies that $e^{ \prn{f_i(x)-f_i(\bx)}/{\epsilon}} = \Theta(1)$ inside the ball, and consequently $\xi_i$ is indeed well-behaved, with Lipschitz constant $O(1)$. We thus minimize $\Gamma_\epsilon$ (and hence $\Fsm$) with stochastic gradient descent~\cite{hazan2014beyond}, sampling $i$ from $p$. Moreover, if $\grad f_i$ are Lipschitz, then $\grad \xi_i$ are also Lipschitz, and we apply an accelerated variance reduction method~\cite{allen2016katyusha} for better efficiency.
		
To complete the analysis of our methods it remains to determine how accurately we need to solve each ball subproblem. Unfortunately, the analysis of \cite{carmon2020acceleration} makes fairly stringent accuracy requirements, and also requires $\grad \Fsm$ to have a finite Lipschitz constant. To obtain tighter guarantees, we significantly rework the analysis in \cite{carmon2020acceleration}, modifying the algorithm to make it applicable  without any differentiablility requirements. Our improved analysis takes into account the fact that the acceleration scheme only requires ball minimization with strong $\ell_2$ regularization, which further improves the oracle implementation  complexity.

Our lower bound follows from a variation on the classical ``chain constructions'' in optimization lower bounds~\cite{nemirovski1983problem,woodworth2016tight,guzman2015lower,diakonikolas2020lower}, where in order to make a unit of progress on our constructed function, any algorithm must (with constant probability) make $\Omega(N)$ queries in order to discover a single new link in the chain.
We build a chain of length $\Omega(\epsilon^{-2/3})$ for which querying any $\epsilon$ minimizer of $\F$ requires discovering the entire chain, giving the $\Omega(N\epsilon^{-2/3})$ complexity lower bound. To prove this result for arbitrary randomized algorithms, we randomize both the order of the functions and the rotation of the domain. 

\paragraph{Paper outline.} 
\Cref{sec:prelims} provides some additional preliminaries and notation. \Cref{sec:ms-bacon-redux} gives our improved derivation of the BOO acceleration method of~\cite{carmon2020acceleration}, and \Cref{sec:boo-implementation} develops a BOO for $\Fsm$, culminating in our upper complexity bounds for the problem~\eqref{eq:problem}, stated in  \Cref{thm:ub}. \Cref{sec:lb} gives our lower bounds with the main result stated in \Cref{thm:lb}.
\section{Preliminaries}\label{sec:prelims}

\paragraph{General notation.} 
Throughout, $\norm{\cdot}$ denotes the Euclidean norm. We write $\ball_r(z)$ for the Euclidean ball of radius $r$ centered at $z$, and $\ball_r^d(z)$ when emphasizing that the ball is $d$-dimensional. We use $\Lf$ to denote a function Lipschitz constant and $\Lg$ to denote a gradient Lipschitz constant; we say that $f$ is $\Lg$-smooth if it has $\Lg$-Lipschitz gradient. 
To disambiguate between sequence and coordinate indices, in \Cref{sec:lb} we denote the former with normal subscript and the latter with bracketed subscript, i.e., $x\coind{i}$ is the $i$th coordinate of $x$ and $x_k$ is the $k$th element in the sequence $x_1, x_2, \ldots$. We also write $v\coind{\le i}$ to denote a copy of $v$ with coordinate $i+1, i+2, \ldots$ set to zero. 
We use $a\wedge b \defeq \min\{a,b\}$ to abbreviate binary minimization. We write the binary indicator of event $A$ as $\indic{A}$.

\paragraph{Complexity model.}
We mainly measure complexity through the number individual function and gradient evaluations required to solve the problem~\eqref{eq:problem}. We write $\Tf$ for the cost of evaluating $f_i(x)$ for a single $i$ and $x$, and similarly write $\Tg$ for the cost of evaluating $\grad f_i(x)$. Assuming $\Tf, \Tg = \Omega(d)$, our evaluation complexity upper bounds translate directly to runtime upper bounds.

\paragraph{Proximal operators.}
For any function $f$ and regularization parameter $\lambda \ge 0$, we define the standard proximal mapping $\prox{f}(\bx) \defeq \argmin_{x\in\R^d} \crl*{
	f(x) + \frac{\lambda}{2} \norm{x-\bar{x}}^2}$. We also define the ball constrained proximal mapping 
$\bprox{f}(\bx) \defeq \argmin_{x\in\ball_r(\bx)} \crl*{
	f(x) + \frac{\lambda}{2} \norm{x-\bar{x}}^2}$. Finally, we define the notion of an approximate oracle for $\bprox{f}$, which plays a key role in our analysis.

\begin{definition}[\BOO]\label{def:broo}
	We say that a mapping $\oracle{\cdot}$ is a Ball Regularized Optimization Oracle of radius $r$ ($r$-\BOO) for $f$, if for every query point $\bx$, regularization parameter $\lambda$ and desired accuracy $\delta$, it return $\tilde{x} = \oracle{\bx}$ satisfying
\begin{equation}\label{eq:broo-req}
		f(\tilde{x})+\frac{\lambda}{2}\|\tilde{x}-\bx\|^2\le \min_{x \in \ball_{r}(\bx)} \crl*{
		f(x) + \frac{\lambda}{2} \norm{x-\bar{x}}^2}+\frac{\lambda}{2}\delta^2.
\end{equation}
\end{definition}
\noindent
Note that when $f$ is convex, the strong convexity of  $f(x)+\frac{\lambda}{2}\norm{x-\bx}^2$ and the approximation requirement~\eqref{eq:broo-req} guarantee that $\norm{\oracle{\bx}-\bprox{f}(\bx)}\le \delta$.

\section{\BOO acceleration}\label{sec:ms-bacon-redux}
In this section, we describe a variant of the ball optimization acceleration scheme of \citet{carmon2020acceleration}, given as \Cref{alg:ms-bacon-redux}. 
Both methods follow the template of Monteiro-Svaiter acceleration~\cite{monteiro2013accelerated}, but our algorithm improves on  \cite{carmon2020acceleration} in two ways. First, it accesses the objective strictly through the ball oracle, while \cite{carmon2020acceleration} also uses gradient computations. Second, our algorithm requires an oracle that solves \emph{regularized} ball optimization problems, which are easier to implement.\footnote{We note that $\lambda$ in our notation corresponds to $1/\lambda$ in the notation of~\cite{carmon2020acceleration}.}

As a consequence of these differences, our accelerated algorithm's guarantee does not require any smoothness of the objective function. Moreover, our setup allows for far less accurate solutions to the ball optimization subproblems: \citet{carmon2020acceleration} require $\delta = O(\frac{\epsilon}{\Lg R})$ while we only require $\delta = O(\frac{\epsilon}{\lambda R})$. 
While our requirement becomes stricter as the regularizer $\lambda$ grows, it also becomes easier to fulfill since the ball optimization problem becomes more strongly convex and hence easier to solve. Our relaxed accuracy requirement ultimately translates to an improved $\epsilon^{-1}$ dependence in the sublinear-in-$N$ term in our upper bound.

With the key innovations of~\Cref{alg:ms-bacon-redux} explained, we now formally state its convergence guarantee; we defer the proof to  \Cref{sec:ms-bacon-redux-proofs}.

\newcommand{\bisectionTarget}{\Delta}
\newcommand{\idealBisectionTarget}{\widehat{\Delta}}
\newcommand{\xret}{x_{\mathrm{ret}}}

\begin{algorithm2e}[htb!]
\setstretch{1.1}
\caption{\BOO acceleration}
\label{alg:ms-bacon-redux}
\LinesNumbered
\DontPrintSemicolon
\SetKwRepeat{Do}{do}{while}
\KwInput{Initial $x_0\in \R^d$,  Lipschitz and distance bounds $\Lf$, $R$, $r$,  accuracy $\epsilon$, \BOO $\oracle{\cdot}$ %
}
\KwOutput{$\xret$ such that $f(\xret) -  \argmin_{z \in \ball_R(x_0)} f(z) \leq \epsilon$ 
}
\SetKwFunction{Bisection}{$\linesearch$}
\SetKwProg{Fn}{function}{}{}

\BlankLine

Let $v_0 = x_0$, $A_0 = 0$\;
\For{$t=0,1,2,\ldots$}{
	
	$\lambda_{t+1} =
	\linesearch(x_t,v_t,A_t, \lambda_{\max}=\tfrac{2\Lf}{r}, \lambda_{\min}=\tfrac{\eps}{6rR})$ \label{line:lambda_min}\;
	
	\Comment*[f]{Finds $\lambda_{t+1}$ such that $x_{t+1}\approx \prox[\lambda_{t+1}]{f}(y_t)$ and either
		$\norm{x_{t+1}-y_t} \approx  r$ or $x_{t+1}$ is $\epsilon$-optimal}
	
	$a_{t+1} = \tfrac{1}{2\lambda_{t+1}}(1 + \sqrt{1 + 4 \lambda_{t+1} A_t})$ and  $A_{t+1} = A_t + a_{t+1}$ \label{line:at-def} \Comment*{$A_{t+1} = a_{t+1}^2 \lambda_{t+1}$}
	
	$y_{t}=\frac{A_t}{A_{t+1}}x_{t}+\frac{a_{t+1}}{A_{t+1}}v_{t}$\;
	
	$x_{t+1} = \oracle[\lambda_{t+1},\delta_{t+1}]{y_t}$, where $\delta_{t+1} = \frac{\eps}{12 \lambda_{t+1} R}$\;	
	
	$v_{t+1}= \argmin_{v \in \ball_R(x_0)}\left\{a_{t+1} \lambda_{t+1}\left\langle y_{t}-x_{t+1},v \right\rangle +\frac{1}{2}\norm{ v-v_{t}} ^{2}\right\}$
	\label{line:v-update}
	\;
	
	\If{$A_{t + 1} \geq \frac{R^2}{\eps}$, $\lambda_{t + 1} \leq \frac{\eps}{3 r R}$, $\norm{x_{t +1} - v_{t + 1}} > 2R$, \textbf{\textup{or}}  $A_{t + 1} < \exp\left(\frac{r^{2/3}}{R^{2/3}}(t-1)\right) A_1$ \label{line:outer_check}}{
		\Return $\xret \in \argmin_{x\in\crl{x_0,x_1,\ldots,x_{t+1}}} f(x) $ \label{line:outer_return}\;
	}
}

\BlankLine

	\Fn{$\linesearch(x,v,A,\lambda_{\max}, \lambda_{\min})$ %
	}{
	For all $\lambda'$, let $y_{\lambda'} \defeq 
	\alpha_{2A\lambda'} \cdot x + (1-\alpha_{2A\lambda'}) \cdot v$, where
	$\alpha_\tau \defeq \frac{\tau}{1+\tau+\sqrt{1 + 2\tau}}$ 
	
	Define $\bisectionTarget(\lambda) \defeq \norm{\oracle[\lambda, \frac{r}{17}]{y_\lambda} - y_\lambda}$ 
	\Comment*{approximation of $\idealBisectionTarget(\lambda) \defeq \norm{ \bprox[\lambda,r]{f}(y_\lambda) - y_\lambda}$}

	Let $\lambda = \lambda_{\max}$ 
	
	\lWhile{$\lambda \geq \lambda_{\min}$ \textup{\textbf{and}}  $\bisectionTarget(\lambda) \leq \frac{13r}{16}$  \label{line:while1start}}{
		$\lambda \gets \lambda/2$
		\Comment*[f]{terminates in $O(\log\frac{\lambda_{\max}}{\lambda_{\min}})$ steps}
	}\label{line:while1end}
	\lIf{$\lambda \leq \lambda_{\min}$}{\Return $2\lambda$ \label{line:ls_lower_boundary}
	\Comment*[f]{happens only if $\bprox[2\lambda,r]{f}(y_{2\lambda})$ is $O(\epsilon)$-optimal for small $\lambda_{\min}$}
}
	
	Let $\lambda_u = 2\lambda$, $\lambda_{\ell} = \lambda$ and $\lambda_m = \sqrt{\lambda_u \lambda_{\ell}}$
	
	\lIf{$\bisectionTarget(\lambda_{\ell}) \leq  \frac{15r}{16}$}{\Return $\lambda_{\ell}$
	\Comment*[f]{happens only if $\bisectionTarget(\lambda_{\ell}) \in [\frac{13r}{16}, \frac{15r}{16}]$}
	}\label{line:stop-middle}
		
	\While{$\bisectionTarget(\lambda_m)\notin [\frac{13r}{16}, \frac{15r}{16}]$ \textup{\textbf{and}} $\log_2 \frac{\lambda_u}{\lambda_{\ell}} \ge  \frac{r}{8(R+\Lf/\lambda_{\ell})}$ \label{line:while2start} } {
	\leIf{
		$\bisectionTarget(\lambda_m) < \frac{13r}{16}$}{$\lambda_u=\lambda_m$}{$\lambda_\ell = \lambda_m$}
	
	$\lambda_m = \sqrt{\lambda_u \lambda_{\ell}}$
} 	\label{line:while2end} 

	\Return $\lambda_m$ \Comment*[f]{
		the while loop terminates in $O\prn[\big]{\log\prn[\big]{\frac{R}{r} + \frac{\Lf}{\lambda_{\min} r}}}$ steps}
	
}
\end{algorithm2e}

\newcommand{\fancyind}[1]{_{(#1)}}
\begin{restatable}{theorem}{thmMSBaconRedux}\label{thm:ms-bacon-redux}
Let $f:\R^d\to \R$ be convex and $\Lf$-Lipschitz, and let $z \in \R^d$. For any domain bound $R>0$, ball radius $r\in(0,R]$, accuracy level $\epsilon>0$, and initial point $x_0\in\R^d$, \Cref{alg:ms-bacon-redux} returns a point $x\in\R^d$ satisfying $f(x)-\min_{z\in\ball_R(x_0)}f(z) \le \epsilon$ using at most
\begin{equation*}
	T = O\prn*{
		\prn*{ \frac{R}{r}}^{2/3}
		\log  \prn*{\frac{[f(x_0) - \min_{z\in\ball_R(x_0)}f(z)] R}{ \epsilon r} }
		 \log \prn*{\frac{\Lf R^2}{\epsilon r}}
	}
\end{equation*}
queries to an $r$-\BOO. Moreover, the \BOO query parameters  $(\lambda\fancyind{1}, \delta\fancyind{1}), \ldots, (\lambda\fancyind{T}, \delta\fancyind{T})$ satisfy
\begin{enumerate}
	\item \label{item:lambda-bounds} $\Omega(\frac{\epsilon}{rR}) \le \lambda\fancyind{i} \le O(\frac{\Lf }{r})$  and $\delta\fancyind{i} \ge \Omega( \frac{\epsilon}{\lambda\fancyind{i} R} )$ for all $i \in [T]$.
	\item \label{item:lambda-sum-bound} $\sum_{i \in [T]} \frac{1}{\sqrt{\lambda\fancyind{i}}} \leq O\prn[\big]{ \frac{R}{\sqrt{\eps}}\log\frac{\Lf R^2}{\eps r}}.$ %
\end{enumerate}
\end{restatable}
We remark that Theorem~\ref{thm:ms-bacon-redux} requires a bound on the Lipschitz constant of $f$ solely to bound the complexity of the bisection procedure for finding $\{\lambda_{t}\}$. 
\section{\BOO implementation}\label{sec:boo-implementation}
In this section, we develop efficient \BOO implementations for $\Fsm$, the softmax approximation of $\F$~\eqref{eq:softmax}. In \Cref{ssec:approx} we develop our main analytical tool in the form of an ``exponentiated softmax'' function approximating $\Fsm$ and facilitating efficient stochastic gradient estimation. We then minimize the exponentiated softmax with standard tools from stochastic convex optimization. In  \Cref{ssec:sgd} we give a \BOO implementation for the non-smooth case using restarted SGD~\cite{hazan2014beyond}. In \Cref{ssec:asvrg} we instead apply an accelerated variance reduction method (Katyusha~\cite{allen2016katyusha}) that offers improved performance when the $f_i$ are even slightly smooth. Finally, in \Cref{ssec:ub-statement} we combine our \BOO implementations with \Cref{alg:ms-bacon-redux} and its guarantees to obtain our main results: new convergence guarantees for minimizing $\F$. We defer proofs to \Cref{sec:boo-impl-proofs}.

\subsection{Exponentiating a softmax}\label{ssec:approx}
Recall that $\epsilon' = \epsilon/(2\log N)$ and that (for nominal accuracy $\epsilon$) the softmax function $\Fsm(x) = \epsilon' \log\prn*{ \sum_{i\in[N]} e^{f_i(x)/\epsilon'}}$ approximates $\F$ to within $\epsilon/2$ additive error.
The key challenge in designing an efficient stochastic method for minimizing $\Fsm$ is a lack of cheap unbiased gradient estimators. Specifically, we have 
$\grad \Fsm(x) = \sum_{i\in [N]} p_i(x) \grad f_i(x)$, where
\begin{equation}\label{eq:softmax-prob}
	p_i(x) = \frac{e^{f_i(x)/\epsilon'}}{\sum_{j\in [N]} e^{f_j(x)/\epsilon'}}.
\end{equation}
Given access to $p(x)$, we could easily obtain an unbiased estimator for $\grad \Fsm(x)$ by sampling $i\sim p(x)$ and outputting $\grad f_i(x)$. However, computing $p(x)$ itself requires evaluating all $N$ functions, making it basically as costly as computing $\grad \Fsm$ exactly.

This difficulty, however, is greatly relieved when we operate in a small ball of radius $\reps = \epsilon' / \Lf$ centered at some point $\bx$. To see why, note that for every $i$ and every  $x\in\ball_{\reps}(\bx)$, Lipschitz continuity of $f_i$ implies  $| f_i(x)/\epsilon' -f_i(\bx)/\epsilon' | \le  \Lf \reps / \epsilon' = 1$. Consequently, $p(\bx)$ is a multiplicative approximation for $p(x)$ throughout the ball, satisfying $e^{-2} p_i(\bx) \le p_i(x) \le e^2 p(\bx)$ for all $x\in \ball_{\reps}(\bx)$. Our high-level strategy is thus: perform a full data pass \emph{once} to compute $p(\bx)$, and then rely on the stability of $p(x)$ within $\ball_{\reps}(\bx)$ to efficiently estimate gradients by sampling from $p(\bx)$.
However, simply sampling $i\sim p(\bx)$ and returning $\grad f_i(x)$ is not enough, because it leads to a biased estimator of $\grad \Fsm(x)$. Instead, we define below a surrogate function ``exponentiating the softmax''  that closely approximates $\Fsm$ and for which $e^{(f_i(x)-f_i(\bx))/\epsilon'} \grad f_i(x)$ is an unbiased gradient estimator when $i\sim p(\bx)$.\footnote{We remark that $g_i(x) = e^{(f_i(x)-f_i(\bx))/\epsilon'} \grad f_i(x)$ is also nearly unbiased for $\Fsm$ in the sense that $\E g_i(x) = Z(x) \grad \Fsm(x)$ for some $Z(x)$ that is close to 1 when inside $\ball_{\reps}(\bx)$. Estimators of this form suffice for SGD, but are less amenable to variance reduction.}

To precisely define the surrogate ``exponentiated softmax'' function, we require some additional notation. Fixing a ball center $\bx$ and regularization parameter $\lambda$, let
\begin{equation*}
	f_i^\lambda(x) \defeq f_i(x) + \frac{\lambda}{2}\norm{x-\bx}^2
	~~\mbox{and}~~
	\Fsm^\lambda(x) \defeq \Fsm(x) + \frac{\lambda}{2}\norm{x-\bx}^2
	= \epsilon' \log\prn*{ \sum_{i\in[N]} e^{f_i^\lambda(x)/\epsilon'}}
\end{equation*}
be the regularized counterparts of $f_i$ and $\Fsm$, respectively. Then, we define the exponentiated softmax as
\begin{equation}
	\label{eq:bare-exp-def}
	\beF(x) = {\epsilon'} \cdot \exp\prn*{\frac{\Fsm^{\lambda}(x)-\Fsm^{\lambda}(\bx)}{{\epsilon'}}} =  \sum_{i\in [N]} p_i(\bx) \gamma_i(x)~\mbox{where}~
	\gamma_i(x)\defeq \epsilon' e^{ \frac{f_i^\lambda(x)-f_i^{\lambda}(\bx)}{{\epsilon'}}}.
\end{equation}
Clearly, $\beF$ is a finite sum objective (weighted by $p(\bx)$), making stochastic first-order methods applicable. Moreover, as the following lemma shows, when the ball radius $r$ and $\lambda$ are not too large, $\beF$ closely approximates $\Fsm^\lambda$ and is as regular as $\Fsm^\lambda$ up to a constant.

\begin{restatable}{lem}{beFapprox}\label{lem:bef-approx}
Let $f_1,\cdots, f_N$ each be $\Lf$-Lipschitz and $\Lg$-smooth gradients. For any $c>0$, $r \le c \eps' / \Lf$, and $\lambda \le c\Lf/r$ let $C = (1+c+c^2)e^{c+c^2/2}$. The exponentiated softmax $\beF$ satisfies the following properties for any $\bx\in\R^d$.
\begin{enumerate}
\item \label{item:bef-error-bound} $\Fsm^{\lambda}(x)$ and $\beF$ have the same minimizer $\xopt$ in $\ball_r(\bx)$. Moreover, for every $x\in\ball_r(\bx)$,
\[\Fsm^{\lambda}(x) - \Fsm^{\lambda}(\xopt) \le C ( \beF(x) - \beF(\xopt) ).\]
\item  \label{item:bef-reg-bound} Restricted to $\ball_r(\bx)$, each function   $\gamma_i$ defined in~\eqref{eq:bare-exp-def} is $C \Lf$-Lipschitz, $C^{-1} \lambda$ strongly convex, and $C(\Lg + \lambda + \Lf^2 / \epsilon')$-smooth.
\end{enumerate}
\end{restatable}
\noindent
The proof of \Cref{lem:bef-approx} follows from a straightforward calculation, and we defer it to \Cref{ssec:bef-approx-proof}.

\subsection{The non-smooth case: SGD implementation}\label{ssec:sgd}

To take advantage of the strong convexity of of $\beF$ we use the restarted SGD variant of~\citet{hazan2014beyond}, which finds an $\varepsilon$-suboptimal point of a $G$-Lipschitz and $\mu$-strongly convex function with $\Otil{G^2 / (\mu \varepsilon)}$ iterations (with high probability). To estimate the stochastic gradients, we sample $i\sim p(\bx)$ and output $\grad \gamma_i(x)$; this takes $O(\Tg + \Tf)$ time per stochastic gradient, plus $O(N\Tf)$ preprocessing time to compute $p(\bx)$.
We provide pseudocode for the algorithm in \Cref{ssec:sgd-app}, where we also prove the following complexity bound.

\begin{restatable}{corollary}{corsgd}\label{cor:sgd}
	Let $f_1,f_2,\cdots,f_N$ be $L_f$ Lipschitz, let $\sigma\in(0,1)$, $\epsilon, \delta > 0$ and $\reps =  \eps/(2\log N\cdot\lip)$. 
	For any $\bx\in\R^d$ and $\lambda \le O(\Lf / \reps)$, with probability at least $1-\sigma$, \Cref{alg:innerloop-SGD} outputs a valid $\reps$-\BOO response for $\Fsm$ to query $\bx$ with regularization $\lambda$ and accuracy $\delta$, and has  cost
\begin{equation}\label{eq:sgd-complexity-bound}
		O\left(\Tf N+ (\Tg + \Tf) \frac{\Lf^2}{\lambda^2\delta^2}\log\left(\frac{\log(\Lf/\lambda\delta)}{\sigma}\right)\right).
\end{equation}
\end{restatable}

\subsection{The (slightly) smooth case: accelerated variance reduction implementation}\label{ssec:asvrg}

If we further assume smoothness of $f_1, \ldots, f_N$, we can use stochastic variance reduction to obtain an improved runtime. With these methods, we estimate the gradient of $\beF$ as $\grad \beF(x') + \grad \gamma_i(x) - \grad \gamma_i(x')$, where $i\sim p(\bx)$ and $x'$ is a reference point which we recompute $\Otil{1}$ times. Here, the $O(N\Tf)$ cost of computing $p(\bx)$ is essentially free compared to the cost $\Otil{N\Tg}$ of computing the exact gradients of $\beF$ at the reference point. We again take advantage of the regularization-induced $\lambda$-strong-convexity a variant of the Katyusha method of \citet{allen2016katyusha}. This results in the following complexity guarantee; see~\Cref{ssec:asvrg-app} for a proof.

\begin{restatable}{corollary}{corasvrg}\label{cor:asvrg}
Let $f_1$, $\cdots$, $f_N$ be $\Lf$-Lipschitz and $\Lg$-smooth,  let $\sigma\in(0,1)$, $\epsilon, \delta > 0$, $\epsilon' =  \epsilon / (2\log N)$ and $\reps =\epsilon' / \Lf$. For any $\bx\in\R^d$ and $\lambda \le O(\Lf / \reps)$, with probability at least $1-\sigma$, Katyusha1 \cite{allen2016katyusha} outputs a valid $\reps$-\BOO response to query $\bx$ with regularization $\lambda$ and accuracy $\delta$, and has computational cost
\begin{equation}\label{eq:asvrg-complexity-bound}
O\left(\left(\Tf+\Tg\right)\left(N+{\frac{\sqrt{N}\left(\Lf + \sqrt{\epsilon' \Lg}\right)}{\sqrt{\lambda \epsilon'}}}\right)\log\left(\frac{\Lf\reps}{ \lambda\delta^2\sigma}\right)\right).
\end{equation}
\end{restatable}

\subsection{Main result}\label{ssec:ub-statement}
With our oracle implementations in hand, we are ready to state our main result.
\begin{restatable}{theorem}{thmub}\label{thm:ub}
	Let $f_1,f_2,\ldots,f_N$ be $L_f$-Lipschitz, let $\xopt$ be a minimizer of $\F(x) = \max_{i\in[N]} f_i(x)$ and assume $\norm{x_0-\xopt}\le R$ for a given initial point $x_0$ and some $R>0$. For any $\epsilon>0$, \Cref{alg:ms-bacon-redux} with the \BOO implementation for $\Fsm$ in Algorithm~\ref{alg:innerloop-SGD} solves the problem~\eqref{eq:problem} with probability at least $\frac{99}{100}$ and has computational cost
	\begin{equation}
		\label{eq:final-rate-non-smooth}
		O\left(\left(\frac{\lip R\log N}{\eps}\right)^{2/3}\left( \Tf N+\left(\frac{\Lf R}{\eps}\right)^2\cdot(\Tf+\Tg)\log K\right)\log^2K\right),
	\end{equation}
	where $K \defeq \Lf R\epsilon^{-1}\log N$.
	If moreover $f_1,f_2,\ldots,f_N$ are each $\Lg$-smooth, then \Cref{alg:ms-bacon-redux} with a \BOO implementation for $\Fsm$ using
	 Kayusha1 solves~\eqref{eq:problem} with probability $\ge \frac{99}{100}$ and has  cost
	\begin{equation*}
	\label{eq:final-rate-smooth}
	O\left((\Tf+\Tg) \left(
	\left(\frac{\lip R\log N}{\eps}\right)^{2/3}N
	+\left(  \frac{\Lf R\sqrt{\log N}}{\eps}
	+\sqrt{\frac{\Lg R^2}{\eps}}\right)\sqrt{N}\right)\log^3 K\right).
	\end{equation*}
\end{restatable}

\noindent
The proof of \Cref{thm:ub}, which we provide in \Cref{ssec:ub-proof}, follows straightforwardly from \Cref{thm:ms-bacon-redux} and \Cref{cor:sgd,cor:asvrg}. %
 When applying \Cref{cor:sgd} with $\delta = \Omega(\frac{\epsilon}{\lambda R})$ the dependence of the complexity on $\lambda$ cancels, and we get that each oracle call costs $\Otil{N\Tf + {\Lf^2 R^2}\epsilon^{-2}(\Tf+\Tg)}$. The complexity bound then follows from multiplying the per-call cost with the bound $\Otil{(R/\reps)^{-2/3}}$ that \Cref{thm:ms-bacon-redux} provides on the total number of oracle calls. When applying \Cref{cor:asvrg} we obtain an oracle implementation cost of 
$\Otil{N(\Tf +\Tg) + \lambda^{-1/2}\sqrt{N}\sqrt{\Lf^2 \epsilon^{-1}+ \Lg}(\Tf+\Tg)}$.
The complexity bound again follows by multiplying the per-call cost again with the total number of calls, except that to bound the contribution of $\sqrt{N}$ term we invoke the the guarantee $\sum_{i} \lambda\fancyind{i}^{-1/2} \le \Otil{R\epsilon^{-1/2}}$ in \Cref{thm:ms-bacon-redux} to a tighter bound.
\section{Lower bounds}\label{sec:lb}

\newcommand{\prog}[1][\alpha]{\mathrm{prog}_{#1}}
\newcommand{\tf}{\tilde{f}}
\newcommand{\fhard}{\hat{f}^{\{T,N,\nsm\}}}
\newcommand{\Fhard}{\hat{F}_{\max}^{\{T,N,\nsm\}}}
\newcommand{\nsm}{\ell}
\newcommand{\linkfun}[1][\alpha_T,\nsm]{\psi_{#1}}
\newcommand{\fnull}{f_{\emptyset}}
\newcommand{\problemInstance}{(f_i)_{i\in [N]}}
\newcommand{\orcLocal}{\mathcal{O}^{\mathrm{loc}}}

\newcommand{\B}[1][t]{B_{#1}}
\newcommand{\C}[1][t]{C_{#1}}
\newcommand{\breachEv}[1][t]{\mathfrak{E}_{#1}}
\NewDocumentCommand{\orcInfo}{O{\pi} O{U}}{\mathcal{I}_{#1,#2}}
\newcommand{\bp}[1][t]{\bar{p}_{#1}}
\newcommand{\zrEv}[1][t]{\mathfrak{ZR}_{#1}}
\newcommand{\lbt}{\tau}
\newcommand{\perm}{\Pi}
\newcommand{\iperm}{\Pi^{-1}}
\newcommand{\algseed}{\zeta}
\newcommand{\failedInds}{I^{-}}
\newcommand{\restrictedTo}[1]{\vert_{#1}}

In this section, we prove oracle complexity lower bounds showing that the results of the previous section are order optimal for sufficiently large $N$ and  $\Lg$.
While our algorithms are first-order methods, our lower bounds remain valid even for other algorithms that use high order derivatives, as is typical for our proof technique.

We begin by providing a formal definition of the oracle-based optimization model we consider (\Cref{sec:lb-protocol}). In \Cref{sec:lb-progress-control}, we define an $N$-element variant for the zero-chain concept, and prove that it allows us to control the progress of any (possibly randomized) algorithm. Then, in \Cref{sec:lb-hard-instace} we construct a particular $N$-element zero-chain for which slow progress implies a large optimality gap. Finally, \Cref{sec:lb-statement} ties these results together, giving our lower bound and providing some discussion.

\subsection{Optimization protocol}\label{sec:lb-protocol}
Consider problem instances of the form $\problemInstance$, where $f_i: D\to \R$ for some common domain $D$ and all $i\in[N]$. We say that an algorithm operating on $\problemInstance$  is an  \emph{$N$-element algorithm}
if it uses the following iterative protocol. At iteration $t$, the algorithm produces a query $i_t, x_t$, with $i_t\in [N]$ and $x_t \in D$. It then observes the output of a  \emph{local oracle} for $f_{i_t}$ at the point $x_t$, which we denote by $\orcLocal_{f_{i_t}}(x_t)$. 

Formally, $\orcLocal$ can be any mapping that satisfies $\orcLocal_{f}(x) = \orcLocal_{\tilde{f}}(x)$ whenever $f(y)=\tilde{f} (y)$ for all $y$ in some open set containing $x$ (subsequently referred to as a ``neighborhood'' of $x$). In particular, the first-order oracle used for our upper bounds corresponds to  $\orcLocal_{f}(x) = (f(x), \grad f(x))$ and is valid local oracle. The $p$th order derivative oracle $\orcLocal_{f}(x) = (f(x), \grad f(x), \ldots, \grad^{p} f)$ is also a valid local oracle.
The notion of local oracles is classical in the literature on information-based complexity~\cite{nemirovski1983problem,guzman2015lower}. 

The algorithms we consider may be randomized, and we use $\algseed$ to denote the algorithm's randomness. Beyond $\algseed$, the query of the algorithm at iteration $t$ may only depend on the information it observes from the oracle. That is, for any $t\ge 1$, we have
\begin{equation}\label{eq:general-iter}
	i_t, x_t = Q_t \prn*{ \algseed, \orcLocal_{f_{i_1}}(x_1), \ldots, \orcLocal_{f_{i_{t-1}}}(x_{t-1}) }  
\end{equation}
for some measurable function $Q_t$.

\subsection{Progress control argument}\label{sec:lb-progress-control}

Following well-established methodology~\cite{nesterov2018lectures,guzman2015lower,carmon2020lower}, instead of directly bounding the sub-optimality of the queries $x_1, \ldots, x_t$ we first bound a surrogate quantity we call \emph{progress}. Informally, the progress is the highest coordinate index that the algorithm managed to ``discover'' using the oracle responses. Formally, we define the progress of a point $x$ as
\begin{equation}
	\label{eq:progress}
	\prog[\alpha](x) \defeq \max\crl*{i\ge 1 \;\big|\; \abs{x\coind{i}} > 
		\alpha}~~\mbox{(where $\max \emptyset \defeq 0$)}.
\end{equation}
The parameter $\alpha$ is a significance threshold for declaring a coordinate ``discovered;'' it allows us to prevent algorithms from trivially discovering coordinates by querying directions at random.

We next define a structural property that facilitates controlling the rate with which $\prog(x_t)$ increases. For this definition, we recall that $v\coind{\le l}$ denotes the vector whose first $l$ coordinates are identical to those of $v$ and the remainder are zero. Recall also that $\ball_1^T(0)$ is the unit ball in $\R^T$.

\begin{definition}\label{def:zc}
	A sequence $f_1,\ldots,f_N$ of functions $f_i:\ball_1^T(0)\to \R$ is called an \emph{$\alpha$-robust $N$-element zero-chain} if for all $x \in \ball_1^T(0)$, all $y$ in a neighborhood of $x$, and all $i \in [N]$, we have
	\begin{equation}\label{eq:zc-def}
			\prog(x) \le p \implies f_i(y) = 
			\begin{cases}
							f_i(y\coind{\le p}) &     i < p + 1  \\ 
							f_i(y\coind{\le p+1}) & i = p + 1 \\ 
							f_{N}(y\coind{\le p}) & i > p+1. \\ 
			\end{cases}
	\end{equation}
\end{definition}

To unpack this definition, consider any first-order algorithm with the following two simplifying properties: (1) the queries $i_1, i_2, \ldots$ are drawn i.i.d.\  from 
 $\mathrm{Uniform}([N])$ and (2) every query $x_t$ lies in the span of previously observed gradients  $\grad f_{i_1}(x_1), \ldots, \grad f_{i_{t-1}}(x_{t-1})$ \cite[cf.][]{nesterov2018lectures}. The first query of the algorithm must be $x_1=0$, and consequently $\prog(x_1)=0$. \Cref{def:zc} then implies that $f_2,\ldots,f_N$ are all constant in a neighborhood of $x_1$, while $f_1$ depends only on the first coordinate. Therefore, the span of the gradients (and the next query's progress) can only increase to $1$ after the algorithm queries $i=1$ for the first time. With uniformly random index queries, that takes $\Omega(N)$ queries with constant probability. Repeating this argument, we see that every increase of the gradient span (and hence query progress) takes $\Omega(N)$ queries with constant probability, and therefore reaching progress $T$ takes $\widetilde{\Omega}(NT)$ queries with high probability. 

To extend this conclusion to general algorithms of the form~\eqref{eq:general-iter}, we perform two types of randomization. First, to handle arbitrary strategies for choosing  $i_t$ (as opposed to uniform sampling), we apply a random permutation to $f_1, \ldots, f_N$. Second, to handle arbitrary queries $x_t$ (as opposed to queries in the span of observed gradients), we randomly rotate the coordinate system. This randomization scheme guarantees that no algorithm can materially improve on uniform sampling and span-preserving, as we formally state in the following.

\begin{restatable}{prop}{propProgControl}\label{prop:prog-control}
	Let $\delta, \alpha \in (0,1)$  and let $N,T\in \N$ with $T\le N/2$. Let $\problemInstance$ be an $\alpha$-robust $N$-element zero-chain with domain $\ball_1^T(0)$. For $d \ge T + \frac{2}{\alpha^2}\log \frac{4NT^2}{\delta}$, draw $U$ uniformly from the set of $d\times T$ orthogonal matrices, and draw $\perm$ uniformly from the set of permutations of $[N]$. Let $\tf_i(x) \defeq f_{\iperm(i)} (U^\top x)$. Let $\{(i_t, x_t)\}_{t\ge 1}$ be the queries of any $N$-element algorithm operating on $\tf_1, \ldots, \tf_N$. Then with probability at least $1-\delta$ we have
	\begin{equation*}
		\prog(U^\top x_t) < T~~\mbox{for all}~~t\le \tfrac{1}{16}N\prn*{T-\log\tfrac{2}{\delta}}.
	\end{equation*}
\end{restatable}

See \Cref{sec:prog-control-proof} for a proof. Our definition of $N$-element zero-chains and our proof of their progress control property builds on the notion of (single element) zero-chain functions~\cite{carmon2020lower}. It is also closely related to probability-$p$ zero-chains~\cite{arjevani2019lower}; \Cref{prop:prog-control} essentially shows that $N$-element algorithms interacting with an $N$-element zero-chain make progress about as slowly as stochastic algorithms interacting with with a probability-$N^{-1}$ zero-chain.

\subsection{Hard instance construction}\label{sec:lb-hard-instace}
With the progress-control machinery in hand, we proceed to constructing a specific $N$-element zero-chain that also guarantees a large optimality gap for points with progress smaller than $T$. Toward that end, we first define the ``link function'' $\linkfun[\alpha,\nsm]: \R \to \R_+$ as 
\begin{equation*}
	\linkfun[\alpha, \nsm] (t) \defeq 
	\begin{cases}
		0 &   |t| \le \alpha \\
		\frac{\nsm}{2} (t-\alpha)^2 & \alpha \le |t| \le \nsm^{-1} + \alpha \\
		|t| - \alpha - \frac{1}{2\nsm} & \mbox{otherwise}.
	\end{cases}
\end{equation*}
Clearly, $\linkfun[\alpha, \nsm]$ is 1-Lipschitz, $\nsm$-smooth, and is identically zero for all $|t|\le \alpha$. We note that $\linkfun[\alpha, \nsm]$ is the composition of the Huber function~\cite{huber1992robust} with $\max\{0,|t|-\alpha\}$. 

Chain constructions of the form $\sum_{i \in[N]} \linkfun(x\coind{i}-x\coind{i-1})$ are common in lower bounds for convex optimization \cite[cf.][]{nesterov2018lectures,woodworth2016tight}. For our construction, we instead spread the link components across the different elements. Formally, for $i\in[N]$, we define the $i$th function in the our hard instance as
\begin{equation}\label{eq:hard-instance-def}
	\fhard_i(x) \defeq 
	\begin{cases}
		\linkfun \prn*{ \frac{x\coind{i} - x\coind{i-1}}{{2}} } & i \le T \\
		0 & \mbox{otherwise}\\
	\end{cases}
	~~\mbox{where}~~\alpha_T \defeq \frac{1}{4 T^{3/2}}~\mbox{and}~x\coind{0}\defeq \frac{1}{\sqrt{T}}.
\end{equation}
The following lemma summarizes the properties of our construction. The proof of the lemma is straightforward and we provide it in \Cref{sec:hard-instance-props-proof}

\begin{restatable}{lem}{lemHardInstanceProps}\label{lem:hard-instance-props}
	For every $T,N\in \N$ and $\nsm \ge 0$, such that $T\le N$, we have that  
	\begin{enumerate}
		\item \label{item:hi-zc} The hard instance $(\fhard_i)_{i\in N}$  is an $\alpha_T$-robust $N$-element zero-chain.
		\item \label{item:hi-lip} The function $\fhard_i$ is 1-Lipschitz and $\nsm$-smooth for every $i\in\ [N]$.
		\item \label{item:hi-optgap} For $x\in \domain$ with  $\prog[\alpha_T](x) < T$, the objective $\Fhard(x) = \max_{i\in [N]} \fhard_i(x)$ satisfies
\begin{equation*}
			\Fhard(x)- \min_{x\opt\in\ball_1(0)} \Fhard(x\opt)  \ge \linkfun\prn*{\frac{3}{8T^{3/2}}} \ge
			\min\crl*{ \frac{1}{8T^{3/2}}, \frac{\nsm}{32T^3} }.
\end{equation*}
	\end{enumerate}
\end{restatable}

\subsection{Lower bound statement}\label{sec:lb-statement}

Finally, we combine the results of the previous sections to state our lower bound. In the statement, we use $a\wedge b \defeq \min\{a,b\}$ to abbreviate binary minimization.

\begin{restatable}{theorem}{thmLB}\label{thm:lb}
	Let $\Lf, \Lg, R > 0$, $\epsilon < \Lf R \wedge \Lg R^2$, $N\in\N$ and $\delta\in (0,1)$. Then, for any (possibly randomized) algorithm there exists an $\Lf$-Lipschitz and $\Lg$-smooth functions $(f_i)_{i\in[n]}$ with domain $\ball_R^{d}(0)$ for $d=O \prn*{ \brk*{
			\prn[\big]{\frac{\Lf R}{\epsilon}}^{2} \wedge \prn[\big]{\frac{\Lg R^2}{\epsilon}}}
			 \log \frac{N(\Lf R \wedge \Lg R^2)}{\epsilon } }$ such that with probability at least $\half$ over the randomness of the algorithm, the first
	\begin{equation}
		\label{eq:final-lower-bound}
		\Omega \prn*{  
			N \brk*{ \prn[\Big]{\frac{\Lf R}{\epsilon}}^{2/3} \wedge \prn[\Big]{\frac{\Lg R^2}{\epsilon}}^{1/3}}
			+
			\brk*{ \prn[\Big]{\frac{\Lf R}{\epsilon}}^2 \wedge \prn[\Big]{\frac{N \Lg R^2}{\epsilon}}^{1/2}}	
		} 
	\end{equation}
	queries of the algorithm are all $\epsilon$-suboptimal for $\F(x) = \max_{i\in [N]} f_i(x)$.
\end{restatable}

See \Cref{sec:lb-proof} for a proof of this result. The first  (linear-in-$N$)  term in the lower bound follows from \Cref{prop:prog-control} and \Cref{lem:hard-instance-props} via a re-scaling argument.  The second (sublinear-in-$N$) lower bound term is a direct consequence of existing lower bounds~\cite{diakonikolas2020lower,woodworth2016tight,fang2018near}.

We remark that our lower bound is stated for optimization constrained to a ball of radius $R$, while our upper bounds assume unconstrained optimization given a minimizer of norm at most $R$. These two settings are essentially equivalent; in \Cref{app:lb-unconstrained} we sketch a general technique for transferring lower bounds to the unconstrained setting.

In \Cref{table:summary} we specify our lower bound in the special cases $\Lg=\infty$ and $\Lg=\Theta(\Lf^2 / \epsilon)$, showing that they match our upper bounds (up to polylogarithmic factors) for $N=\Omega( (\Lf R / \epsilon)^2 )$ in the former case and for any $N$ in the latter. More broadly, when  $\Lg=\Theta(\Lf^{2+q} R^{q} / \epsilon^{1+q} )$ our lower and upper bounds match for any $N$ and $q\in[0,2/3]$. For $\Lg = o(\Lf^2 /\epsilon)$ and $\Lg = \omega(\Lf^{8/3}R^{2/3}/\epsilon^{5/3})$, however, there remain  gaps between our upper and lower bounds. We discuss these gaps in the following section. 

\section{Discussion}\label{sec:discussion}
To conclude the paper, we provide some commentary on our results and the possibilities of improving them. For simplicity, in this section we revert to the setting $\Lf=R=1$ used in the introduction. We also use $a \ll b$ as a shorthand for $a = O(b)$, and ignore constant and logarithmic factors throughout.

\subsection{Gaps between the upper and lower bounds}

\paragraph{Regimes where a gap exists.}
Comparing our upper bound in~\Cref{thm:ub} to our lower bound in~\Cref{thm:lb}, we identify two regimes where our upper and lower bounds disagree by more than polylogarithmic factors. The first is the \emph{smooth regime} $\Lg \ll \epsilon^{-1}$, the lower bound is $\Omega(N\Lg^{1/3}\epsilon^{-1/3} + \sqrt{N\Lg\epsilon^{-1}})$ while our upper bound is $\Otil{N\epsilon^{-2/3} + \sqrt{N}\epsilon^{-1}}$, and a different algorithm gives a better oracle complexity $O(N\sqrt{\Lg\epsilon^{-1}})$ (see \Cref{sec:app-discussion-nesterov}) which still falls short of the lower bound. 

The second regime is the \emph{non-smooth} regime $\Lg \gg \epsilon^{-1}$, where both the upper and lower bounds share the term $N\epsilon^{-2/3}$. Comparing the lower bound to the variance reduced upper bound~\eqref{eq:final-rate-smooth}, we see that they disagree if and only if $N\epsilon^{-2/3} + \epsilon^{-2}  \ll \sqrt{N\Lg \epsilon^{-1}}$ which is equivalent to $N \ll \Lg \epsilon^{1/3}$ and $N \gg \epsilon^{-3}/\Lg$. Clearly, this is only possible only when $\Lg \gg \epsilon^{-5/3}$, and so we conclude that the rate~\eqref{eq:final-rate-smooth} is in fact optimal whenever  $\epsilon^{-1} \ll \Lg \ll \epsilon^{-5/3}$. Moreover, the upper bound~\eqref{eq:final-rate-non-smooth} matches the lower bound whenever $N \gg \epsilon^{-2}$ for any $\Lg\gg \epsilon^{-1}$. We conclude that gaps in the non-smooth regime exist only for $\Lg \gg \epsilon^{-5/3}$ and $\epsilon^{-3} / \Lg \ll N \ll \min\{\epsilon^{-2}, \epsilon^{1/3}\Lg\}$. 

\paragraph{Closing the gap in the non-smooth regime.}
Improving the bound~\eqref{eq:final-rate-non-smooth} from $\Otil{N\epsilon^{-2/3} + \epsilon^{-8/3}}$ to $\Otil{N\epsilon^{-2/3} + \epsilon^{-2}}$ would imply that~\eqref{eq:final-lower-bound} gives the optimal rate for any $\Lg \gg \epsilon^{-1}$. The main barrier for obtaining such improvement is our accuracy requirements $\delta_t = O(\epsilon / \lambda_t)$ in \Cref{alg:ms-bacon-redux}. Meeting this requirement with SGD means that each oracle implementation costs $\Otil{N+\epsilon^{-2}}$ function/gradient evaluations, and multiplying this cost by the number of rounds $\Otil{\epsilon^{-2/3}}$ yields the exponent $8/3$. A variant of \Cref{alg:ms-bacon-redux} which can handle less accurate \BOO outputs could close this gap by allowing a more efficient SGD-based implementation.

\paragraph{Closing the gap in the smooth regime.}
The gap between our upper and lower bounds when $\Lg \ll \epsilon^{-1}$ is more fundamental than the one arising for $\Lg \gg \epsilon^{-5/3}$, because it affects the term linear in $N$. The barrier for improving the linear term in our algorithm is the ball radius. Any $\reps$-\BOO implementation with $\Omega(N)$ cost will have overall complexity $\Omega(N\reps^{-2/3})$. The techniques we develop in \Cref{sec:boo-implementation} only allow us to support $\reps=\Otil{\epsilon}$, because this is the largest radius where the exponentiated softmax is stable (see~\Cref{lem:bef-approx}). 

\paragraph{Conjectures and future work.}
We conjecture that our lower bound is in fact optimal in both smoothness regimes. In future work we will attempt to close the remaining complexity gaps described above.

\subsection{Some necessary algorithmic structures}
We now argue that several aspects of our method, namely functions value access, individual function queries and randomization are necessary in any method that achieves (or improves on) our complexity bounds.

\paragraph{Function value access.}
It is possible to minimize a convex function $f$ by iterative (sub)gradient evaluations, without access to the value of $f$ itself. In contrast, all algorithms for minimizing $\F = \max_{i\in[N]} f_i(x)$ \emph{must} query the values of the $f_i$'s in addition to their gradients. To see why this is so, consider the case where $f_i(x) = \Pi(i) - x_i$, where $\Pi$ is a random permutation of $[N]$ and the domain is the unit Euclidean ball. The global minimum of $\max_{i\in [N]}f_i(x)$ is the $\Pi^{-1}(N)$-th standard basis vector. However, gradients provide no information about $\Pi^{-1}(N)$, since $\nabla f_i(x) = -e_i$ for all $x$, independent of $\Pi$.

\paragraph{Individual-function access.}
The algorithms from prior work in \Cref{table:summary} (namely the subgradient method, AGD on softmax and AGD on linearization) are full-batch methods: they proceed by querying all $N$ functions $f_1, \ldots, f_N$  at the same point $x_t$ and using the result to generate the next query point $x_{t+1}$. In contrast, our \BOO implementations proceed by sampling an index $i_t$, computing $\grad f_i$ at $x_{t}$ (and potentially another point), and generating the next query $x_{t+1}$. Full-batch methods are more amenable to parallelization, but for our problem have demonstrably worse oracle complexity. To see this, consider the case where all the $f_i$'s are identical and equal to standard hard instance for convex optimization. For such input, any full-batch methods will have oracle complexity $\Omega(N\min\crl{\epsilon^{-2}, \sqrt{\Lg \epsilon^{-1}}})$~\cite{diakonikolas2020lower}, which is worse than our upper bounds for any $\Lg \gg \epsilon^{-1/3}$ and sufficiently large $N$.

\paragraph{Randomization.}
Another contrast between the prior algorithms in \Cref{table:summary} and our algorithm is that the former are deterministic while ours is randomized. \citet{woodworth2016tight} prove a lower bound of $\Omega(N\min\crl{\epsilon^{-2}, \sqrt{\Lg \epsilon^{-1}}})$ gradient queries for any deterministic method for minimizing the \emph{average} of $N$ functions. Observing that the maximum of the $N$ functions in their construction has the same minimum value as their average (and that the maximum upper bounds the average in any other points), we conclude that this lower bound is also valid for any deterministic method for solving the problem~\eqref{eq:problem}. Therefore, randomization is necessary for obtaining our improved rates of convergence.

\subsection{Practical considerations}
The main purpose of the algorithms we develop in this paper is to clarify the complexity of the fundamental optimization~\eqref{eq:problem}. Nevertheless, since this problem formulation is relevant for a number of machine learning tasks~\cite{clarkson2012sublinear,hazan2011beating,shalev2016minimizing}, it is interesting to try and develop a more practical variant of algorithms. Two aspects of our method which we believe will be particularly useful in practice are the gradient estimation scheme we use in \Cref{alg:innerloop-SGD} and the momentum scheme in \Cref{alg:ms-bacon-redux}. 

However, a number of aspects of our method seem rather impractical. First, the theory instructs us to constrain subproblem solutions to a very small ball of radius $\reps$ of roughly $\epsilon / \Lf$. Since usually neither $\epsilon$ or $\Lf$ are known in advance, the parameter $\reps$ must be tuned. Moreover, choosing $\reps$ to be small in keeping with the theory would likely mean very slow progress in the early stages of the algorithm. A second impractical aspect  is the bisection stage in \Cref{alg:ms-bacon-redux}: while in theory the bisection only increases complexity by a logarithmic factor, in practice it entails solving a considerable number of sub-problems without making progress. This bisection overhead is an issue with Monteiro-Svaiter acceleration more broadly and a topic of active research~\cite{song2019unified,nesterov2019implementable}.
	 
\newpage

 \arxiv{\section*{Acknowledgment}}
 YC was supported in part by Len Blavatnik and the Blavatnik Family foundation, and the Yandex Machine Learning Initiative for Machine Learning. YJ was supported by Stanford Graduate Fellowship. AS was supported in part by a Microsoft Research Faculty Fellowship, NSF CAREER Award
 CCF-1844855, NSF Grant CCF-1955039, a PayPal research award, and a Sloan Research Fellowship.

\arxiv{\bibliographystyle{abbrvnat}}

\newpage

\appendix
\part*{Appendix}

\section{Additional discussion}\label{app:discuss}

Here we discuss three important points that exceeded the scope of our introduction. First, \Cref{sec:app-discussion-smoothness} explains what makes $\Lf^2/\epsilon$ a natural smoothness scale and when we can expect the $f_i$ to be at least that smooth. Then, \Cref{sec:app-discussion-linear} discusses the maximally smooth case of linear loss functions, and compare our results to guarantees of methods specialized to this setting. Finally, \Cref{sec:app-discussion-nesterov} considers the computational complexity of implementing the steps in the accelerated iterative linearization scheme of~\cite[Section 2.3.1]{nesterov2018lectures}.

\subsection{The generality of the smoothness assumption $\Lg = O(\Lf^2 / \epsilon)$}\label{sec:app-discussion-smoothness}

Let $f$ be a convex, $\Lf$ Lipschitz function. When $f$ is not continuously differentiable, it is still possible to uniformly approximate it with a continuously differentiable function, that moreover has a Lipschitz gradient. More concretely, for every $\varepsilon > 0$, we may consider the infimal convolution of $f$ with a quadratic regularizer (also known as its Moreau envelope):
\begin{equation}\label{eq:moreau-smoothing}
	\tilde{f}(x) = \min_{y \in \R^d} \crl*{ f(y) + \frac{\Lf^2}{2\varepsilon} \norm{x-y}^2 }.
\end{equation}
It holds that $0 \le f(x) - \tilde{f}(x) \le \varepsilon/2$ for all $x\in\R^d$ and moreover that $\tilde{f}$ is $\Lf^2/ \varepsilon$ smooth (i.e., with $\Lf^2/\varepsilon$ Lipschitz gradient)~\cite[see, e.g.,][]{guzman2015lower}. Therefore, if we wish to minimize $f$ to accuracy $\epsilon$, we may choose $\varepsilon = \epsilon$ and replace $f$ with the $O(\Lf^2/\epsilon)$-smooth function $\tilde{f}$.

The computational cost of such replacement depends on the application. In certain cases, smoothed versions of the $f_i$'s have closed-form expression, and we may simply define the problem with them instead. Moreover, in several machine learning applications
 each function $f_i$ is ``simple,'' and querying index $i$ really corresponds to obtaining full access to this function, in which case directly computing~\eqref{eq:moreau-smoothing} might be feasible. 

However, when we are truly restricted to accessing $f_i$ through a gradient and value black box, there will be some instances where computing~\eqref{eq:moreau-smoothing} (and indeed any other smoothing) is substantially more expensive than a single oracle query. To see this, note that the worst-case oracle complexity of minimizing a single non-smooth convex function scales as $\Omega(\Lf^2 R^2 \epsilon^{-2})$, while the complexity of minimizing an $\Lg$-smooth function scales as $O(\sqrt{\Lg R^2 \epsilon^{-1}})$. If $T_\epsilon$ is the worst number of oracle calls required to compute an $O(\epsilon)$-accurate $O(\Lf^2 / \epsilon)$-smooth approximation for any $\Lf$-Lipschitz, we immediately have a general complexity upper bound of $O(T_\epsilon \Lf R \epsilon^{-1})$. Comparing it to the lower bound $\Omega(\Lf^2 R^2 \epsilon^{-2})$ immediately yields that $T_\epsilon = \Omega( \Lf R \epsilon^{-1})$ in the worst case. Therefore, the fact that our algorithm maintains the same leading order dependence on $N$ even for $\Lg=\infty$ is nontrivial and potentially useful.

\subsection{The special case of linear loss functions ($\Lg=0$)}\label{sec:app-discussion-linear}

When the losses are linear (i.e., $f_i(x) = a_i^{\top}x + b_i$) exactly $N$ function value and gradient evaluations suffice to completely identify the problem instance, so the optimal oracle complexity should never be more than $N$. In this setting, then, it is more relevant to discuss the \emph{computational} (runtime) complexity of solving the problem~\eqref{eq:problem}. To simplify the following discussion, we return to the setting in the introduction where each $f_i$ is 1-Lipschitz (i.e., $\grad f_i$ has norm at most 1) and we assume the existence of a minimizer of $\F$ with Euclidean norm at most 1.

In the linear setting, an equivalent form of the problem~\eqref{eq:problem} is \begin{equation*}
	\minimize_{x\in\R^{d}}\max_{p\in\Delta^N} p^\top (A x-b)
\end{equation*}
where $A$ is a matrix whose $i$th row contains $\grad f_i$. Stochastic primal-dual methods are able to take advantage of this matrix structure to obtain cheap unbiased estimates for the gradients of $p^T A x$ with respect to both $x$ and $p$, sampling rows and columns of $A$ respectively. Assuming that reading a row and a column of $A$ takes $O(N+d)$ time, the stochastic primal-dual method has runtime complexity $\Otil{(N+d)\epsilon^{-2}}$ which, for sufficiently large $\epsilon$, is sublinear in the problem size $Nd$~\cite{clarkson2012sublinear}. For lower values of $\epsilon$, a variance reduction technique~\cite{carmon2019variance} has preferable runtime complexity $\Otil{Nd + \sqrt{Nd(N+d)}\epsilon^{-1}}$. 

In comparison, in the linear case our method has runtime complexity $\Otil{ Nd\epsilon^{-2/3} + \sqrt{N} d \epsilon^{-1}}$ assuming that sampling a row takes $O(d)$ time. This improves on the variance reduction method in the somewhat narrow parameter regime $N=\Omtil{d}$ and $d = \Otil{\epsilon^{-2/3}}$. However, we note that our method operates under a strictly weaker access assumption, since it only samples rows and not columns. In scenarios where accessing a column takes $\omega(N)$ time, the relative merit of the methods changes. In the most extreme case where reading a column of $A$ is as expensive as reading the entire matrix (e.g., because the rows are scattered across many devices), the stochastic primal-dual methods become less efficient than exact-gradient counterparts with runtime $\Otil{Nd\epsilon^{-1}}$~\cite{nesterov2005smooth,nemirovski2004prox,nesterov2007dual}, while the runtime of our method is unchanged and always superior. To the best of our knowledge, this is the first guarantee for a stochastic gradient method that improves on exact gradient methods in a first-order oracle model that can only provides rows of $A$.

The literature also considers high-order methods for solving the linear case of  problem~\eqref{eq:problem} to better accuracy but with potentially worse dependence on problem dimension. \citet{bullins2020highly} proposes a fourth-order accelerated regularization method that requires $\Otil{\epsilon^{-4/5}}$ solutions of linear systems of the form $A^\top D A = b$ for a positive diagonal matrix $D$. \citet{carmon2020acceleration} use ball oracle acceleration to obtain an improved method requiring only $\Otil{\epsilon^{-2/3}}$ linear system solutions. They also propose to solve these systems using an efficient first-order method, resulting in a runtime guarantee of $\Otil{Nd\eps^{-2/3}+d^{3/2}\eps^{-5/3}}$. We believe more careful reasoning about the conditioning of each linear system to be solved (as we do in~\Cref{ssec:asvrg}) would improve this guarantee to $\Otil{Nd\eps^{-2/3}+d^{3/2}\eps^{-1}}$ under the assumption that individual entries of $A$ take $O(1)$ time to read. In the linear case, this improves on our result when $d < N$, albeit with stronger matrix access assumptions.

For even higher accuracy, it is possible to express the linear case of problem~\eqref{eq:problem} as a linear program and solve it using interior point methods. The best existing theoretical runtimes for these methods are $\widetilde{O}((Nd+(N\wedge d)^{2})\sqrt{N\wedge d})$~\citep{lee2015efficient} or $\widetilde{O}(Nd+d^{2.5})$~\citep{brand2021minimum}, both depending logarithmically on the desired accuracy $1/\eps$. When the problem  dimensions $N$ and $d$ are sufficiently large compared to $1/\epsilon$, first-order methods are preferable.

\subsection{The computational complexity of accelerated iterative linearization}\label{sec:app-discussion-nesterov}

This subsection uses our full notation defined in~\Cref{sec:prelims}. We suggest considering this section after reading~\Cref{sec:ms-bacon-redux,sec:boo-implementation} as well.

In~\cite[Section 2.3.1]{nesterov2018lectures}, Nesterov shows how solving $O(\sqrt{\Lg R^2 \epsilon^{-1}})$ subproblems of the form
\begin{equation*}
	 \min_{x\in\R^d} \max_{i\in[N]}\crl*{ f_i(y_t) +(\grad f_i(y_t))^\top (x-y_t) + \frac{\Lg}{2}\norm{x-y_t}^2 }
\end{equation*}
allows solving the problem~\eqref{eq:problem} when the functions are $\Lg$-smooth. We note that each subproblem is equivalent to
\begin{equation*}
	\min_{x\in\R^d} \max_{p\in\Delta^N}\crl*{  p^\top (Ax-b)  + \frac{\Lg}{2}\norm{x}^2 }
\end{equation*}
for a matrix $A\in\R^{N\times d}$ whose rows have norm at most $\Lf$. Consequently we may apply a variance-reduced bilinear saddle-point method to solve the subproblem to additive error $\nu$ in time
\begin{equation*}
	\Otilb{ Nd + \sqrt{Nd(N+d)}\frac{\Lf}{\sqrt{\Lg \nu}}},
\end{equation*}
see Proposition 6 and the subsequent discussion in the arXiv version of~\cite{carmon2019variance}.

Applying the same arguments used to prove~\Cref{thm:ms-bacon-redux}---but with $\lambda_t = \Lg$ for all $t$---we have that the required subproblem solution accuracy is $O(\frac{\epsilon^2}{\Lg R^2})$, and consequently we can  solve each problem in time
\begin{equation*}
	\Otilb{ Nd + \sqrt{Nd(N+d)}\frac{\Lf R}{\epsilon}}.
\end{equation*}
Assuming $\Tf+\Tg = \Omega(d)$, the overall cost of the method is
\begin{equation*}
	\Otilb{  N(\Tf + \Tg)\sqrt{\frac{\Lg R^2}{\epsilon}} + \sqrt{Nd(N+d)}\frac{\Lf \sqrt{\Lg} R^2}{\epsilon^{3/2}}}.
\end{equation*}
\section{Proof of Theorem~\ref{thm:ms-bacon-redux}}
\label{sec:ms-bacon-redux-proofs}
In this section we give the analysis of our accelerated algorithm. Our analysis builds off of \cite{carmon2020acceleration} and proceeds in several parts. We first prove several standard technical results  \Cref{sec:ms-bacon-redux-proofs-prelims}. Then, in \Cref{sec:ms-bacon-redux-proofs-main} we give the proof of Theorem~\ref{thm:ms-bacon-redux} assuming the correctness of the $\linesearch$ subroutine. Finally in \Cref{sec:ms-bacon-redux-proofs-bisection} we prove this correctness.  

\subsection{Preliminary technical results}\label{sec:ms-bacon-redux-proofs-prelims}

First, we observe that $\oracle{y}$ returns a point which is close to $\bprox{f}(y)$, the true minimizer of the proximal objective. 
\begin{lem}
	\label{lem:prox_oracle_dist}
	Let $f$ be a convex function and let $x = \bprox{f}(y)$. Then if $\oracle{\cdot}$ is an $r$-BROO for f, the point $\xtilde = \oracle{y}$ satisfies $\norm{\xtilde - x} \leq \delta$. 
\end{lem}

\begin{proof}
	By \Cref{def:broo} of the \BOO, we have
	\[
	\nu =\brk*{f(\xtilde) + \frac{\lambda}{2} \norm{\xtilde - y}^2} - \brk*{f(x) + \frac{\lambda}{2} \norm{x - y}^2} \leq \frac{\lambda\delta^2}{2}.
	\]
	Since the function $f(x) + \frac{\lambda}{2} \norm{x-y}^2$ is $\lambda$-strongly convex, we obtain 
	$\nu \ge \frac{\lambda}{2}\norm{x-\xtilde}_2^2$,
	and substituting $\nu \le \lambda \delta^2/ 2$ gives the result.
\end{proof}

Second, we provide standard facts regarding proximal mappings. Though this follows from standard facts regarding subgradients we provide a self-contained proof for completeness.

\begin{lem}
	\label{lem:grad_mapping}
	Let $f : \R^{d} \rightarrow\R$ be a convex function, $S \subseteq \R^d$ be a closed convex set,  $\lambda \geq 0$, and $x_\lambda, x_0 \in S$ satisfy
	\begin{equation*}
		x_\lambda = \argmin_{x \in S} \crl*{ f(x) + \frac{\lambda}{2} \norm{x - x_0}^2}
	\end{equation*}
	Then $g_\lambda \defeq \lambda(x_0 - x_\lambda)$ is a subgradient of $f$, i.e.,
	\begin{equation}\label{eq:grad_map_1}
		f(y) \geq f(x_\lambda) + \inner{g_\lambda}{y - x_\lambda}
		\text{ for all } y \in S\,.
	\end{equation}
	Further, we have 
	\begin{equation}
			\label{eq:grad_map_1.5}
		g_\lambda^\top (x_\lambda - z) =
	\frac{\lambda}{2} \norm{z - x_0}^2
	- \frac{\lambda}{2} \norm{z - x_\lambda}^2 
	- \frac{\lambda}{2} \norm{x_\lambda - x_0}^2 
	~~\mbox{for all $z \in \R^d$}
	\end{equation}
	and 
	\begin{equation}
		\label{eq:grad_map_2}
	f(x_\lambda) \leq f(y) + \frac{\lambda}{2} \norm{y - x_0}^2 
	- \frac{\lambda}{2} \norm{y - x_\lambda}^2 
	- \frac{\lambda}{2} \norm{x_\lambda - x_0}^2
	\text{ for all } y \in S
	\,.
	\end{equation}
\end{lem}

\begin{proof}	
	Let $x^{\alpha}\defeq\alpha\cdot y + (1-\alpha)\cdot x_\lambda$ for all $\alpha\in (0,1]$. Note that $x^{\alpha} \in S$ for all $\alpha\in (0,1]$ since $S$ is convex and $x_\lambda, y \in S$. Consequently, \eqref{eq:grad_map_1} and convexity of $f$ imply that  for all $\alpha\in[0,1]$ 
	\begin{align*}
		f(x_\lambda) +\frac{\lambda}{2}\norm{x_\lambda - x_0}^{2} & \leq f(x^{\alpha})+\frac{\lambda}{2}\norm{x_0 -x^{\alpha}}^{2} \\
		& \leq\alpha f(y)+(1-\alpha)f(x_\lambda)+\frac{\lambda}{2}\norm{\alpha\left(x_0 - y\right)+(1-\alpha)(x_0 - x_\lambda)}^{2} \,. 
		\end{align*}
	Rearranging yields that for all $\alpha>0$ implies
	\begin{align*}
		f(x_\lambda) - f(y) & \leq\frac{\lambda}{2\alpha}\left[ \alpha^2 \norm{x_0 - y} + 2 \alpha (1 - \alpha) \inner{x_0 - y} {x_0 - x_\lambda}+ (1 - \alpha)^2 \norm{x_\lambda - x_0}^2 - \norm{x_\lambda - x_0}^2
		\right] 
	\end{align*}
	Taking the limit as $\alpha\rightarrow0$ yields that
	\begin{align*}
		f(x_\lambda) - f(y)  
		&\leq \lambda (x_0  -y)^\top (x_0 - x_\lambda) - 2\lambda \norm{x_\lambda - x_0}^2
		= \lambda (x_\lambda - x_0)^\top (x_0 - x_\lambda )
	\end{align*}
	The remaining claims~\eqref{eq:grad_map_1.5} and~\eqref{eq:grad_map_2} follow from direct algebraic manipulation of this inequality and the definition $g_\lambda = \lambda(x_0 - x_\lambda)$.
\end{proof}

Third, we bound the function error induced by proximal mapping. 

\begin{lem}
	\label{lem:prox_props} 
	Let $f : \R^{d} \rightarrow\R$ be a convex function, $\lambda \geq 0$, and $x_\lambda, x_0 \in \R^d$ satisfy $x_\lambda = \prox{f}(x_0)$. If $y \in \R^d$ satisfies $\norm{y - x_0} \leq R$ and $ \norm{x_\lambda - x_0} \le  \Theta$ then $f(x_\lambda) - f(y) \leq \lambda \Theta R$.
\end{lem}

\begin{proof}
Bound \eqref{eq:grad_map_1} in \Cref{lem:grad_mapping} yields 
\begin{align*}
f(x_\lambda) - f(y)
&\leq - \lambda \inner{x_0 - x_\lambda}{y - x_\lambda}
=
- \lambda \inner{x_0 - x_\lambda}{y - x_0}
- \lambda \norm{x_0 - x_\lambda}^2
\\
&\leq \lambda \inner{x_\lambda - x_0}{y - x_0}
\leq \lambda \norm{x_0 - x_\lambda} \cdot \norm{y - x_0}
\leq \lambda \Theta R\,.
\end{align*}
\end{proof}

Fourth, we prove that for any constrained minimizer of $f$ (denoted $\xopt$), \BOO calls either decrease the distance to $\xopt$ or have objective value not much worse than $\xopt$.

\begin{lem}
	\label{lem:dist_monotone}
	Let $f$ be a convex function, $\eps, R \geq 0$,  $x_0 \in \R^d$, $\xopt \in \ball_R(x_0)$, $y \in \ball_R(\xopt)$, and  $x' = \oracle{y}$ for $\delta \leq \frac{\eps}{4\lambda R}$ and $\lambda \geq \frac{\eps}{3 R^2}$. If  $f(x') - f(\xopt) > \frac{\eps}{2}$ and $\prox{f}(y) = \bprox{f}(y)$ then $\norm{x' - \xopt} < \norm{y - \xopt}$.
\end{lem}
\begin{proof}
	Let $x_\lambda = \prox{f}(y)$. Since $\prox{f}(y) = \bprox{f}(y)$,  \Cref{def:broo} of $\oracle{\cdot}$ implies
	\[
	f(x') \le f(x') + \frac{\lambda}{2}\norm{x'-y}^2 \leq f(x_\lambda) + \frac{\lambda}{2} \norm{x_\lambda -y}^2 + \frac{\lambda \delta^2}{2}.
	\]
	Further, \Cref{eq:grad_map_2} implies that
	\[
	f(x_\lambda) \leq f(\xopt) + \frac{\lambda}{2} \norm{\xopt - y}^2 
	- \frac{\lambda}{2} \norm{\xopt - x_\lambda}^2 
	- \frac{\lambda}{2} \norm{x_\lambda - y}^2.
	\]
	Combining these inequalities, rearranging, and using $f(x') - f(\xopt) > \frac{\eps}{2}$ yields
	\begin{equation}\label{eq:prox-distance-decrease-bound}
	\frac{\lambda}{2} \norm{x_\lambda - \xopt}^2 \leq \frac{\lambda}{2} \norm{y- \xopt}^2 - \frac{\eps}{2} + \frac{\lambda\delta^2}{2}.
	\end{equation}
	Since $\lambda \delta^2 \leq \frac{\eps^2}{12 \lambda R^2} \leq \frac{\eps}{6} < \frac{\eps}{2}$, this implies $\norm{x_\lambda - \xopt} \leq \norm{y - \xopt}$. 

	By the triangle inequality, we have $\norm{x' - \xopt} \leq \norm{x_\lambda - \xopt} + \norm{x' - x_\lambda}$. Thus, 
	\begin{flalign*}
			\norm{x' - \xopt}^2  & \leq \norm{x_\lambda - \xopt}^2 + 2 \norm{x_\lambda - \xopt} \cdot \norm{x' - x_\lambda} + \norm{x'- x_\lambda}^2 
			\\&
			\leq \norm{x_\lambda - \xopt}^2 + 2 \delta \norm{y-\xopt} + \delta^2 
			 \leq \norm{x_\lambda - \xopt}^2  + 2 \delta R + \delta^2,
	\end{flalign*}
	 where we have used $\norm{x_\lambda - \xopt} \leq \norm{y - \xopt}$ (argued above) along with  $\norm{x'-\xhat}\le \delta$ (\Cref{lem:prox_oracle_dist}) and the assumption that $y \in \ball_R(\xopt)$. Substituting into~\eqref{eq:prox-distance-decrease-bound}, we have 
	\[
	\lambda \norm{x' - \xopt}^2 \leq  \lambda \norm{y-\xopt}^2 - \eps  + 2 \lambda \delta R + 2 \lambda \delta^2.
	\]
	To conclude the proof we note that $- \eps  + 2 \lambda \delta R + 2 \lambda \delta^2< 0$ 
	since $\lambda \delta \le \epsilon/(4R)$ and $\lambda \delta^2 \le \epsilon/6$.
\end{proof}

Fifth, we bound the movement of the proximal operator by the Lipschitz continuity of the objective.

\begin{lem}\label{lem:prox-movement-bound}
	Let $f:\R^d\to \R$ be convex and $\Lf$-Lipschitz. Then for all $\lambda >0$ and $y\in\R^d$,
	\begin{equation*}
		\norm{\prox{f}(y)-y} \le \frac{\Lf}{\lambda}.
	\end{equation*}
\end{lem}
\begin{proof}
Let $x_\lambda \defeq \prox{f}(y)-y$. Equation \eqref{eq:grad_map_1} of \Cref{lem:grad_mapping} implies that 
\[
f(y) - f(x_\lambda) \geq \lambda \inner{x_0 - x_\lambda}{x_0 -  x_\lambda} = \lambda \norm{x_0 - x_\lambda}^2.
\]
Further, $\Lf$-Lipschtiz continuity of $f$ implies $f(y) \leq f(x_\lambda) + L_f \norm{x_0 - x_\lambda}^2$. Combining and noting that the claim is trivial when $\norm{x_0 - x_\lambda} = 0$ yields the claim.
\end{proof}

Finally, we mention a standard lemma about the relation between the sequences $\{A_t\}$ and $\{\lambda_t\}$ in accelerated proximal methods.
\begin{lem}[{\cite[cf.][Lemma 23]{carmon2020acceleration}}]
	\label{lemma:ak_bounds}
	For any iteration $t$ of Algorithm~\ref{alg:ms-bacon-redux}, we have $A_t  = a_t^2 \lambda_t$ and 
	\[
	\sqrt{A_{t}} \geq \frac{1}{2} \sum_{i\in [t]}\frac{1}{\sqrt{\lambda_i}}.
	\]
\end{lem}
\subsection{Main algorithm analysis}
\label{sec:ms-bacon-redux-proofs-main}

In this section we give the analysis of our accelerated algorithm. Before going into the technical details, let us provide a brief overview of the algorithm and its analysis. At its core, our algorithm is an accelerated proximal point method~\cite{osman92app,lin2015universal,frostig2015regularizing}. 
These methods iteratively compute proximal points of the form  $x_{t+1}\approx \prox[\lambda_{t+1}]{f}(y_t)$ and then use a momentum-like extrapolation scheme to compute $y_{t+1}$. Accelerated proximal points methods differ in  the methods they employ to (approximately) compute that proximal points and the choices of $\{\lambda_t\}$. Our method approximates 
$x_{t+1}\approx \prox[\lambda_{t+1}]{f}(y_t)$ using calls to a \BOO, and employs a bisection procedure that finds values of $\lambda_{t}$ for which such approximation is valid because the ball constraint is inactive, i.e., when $\norm{x_{t+1} - y_t} < r$. 

The crux of the analysis of our method is showing that the optimization error decreases roughly as $\exp\prn[\big]{ - \Omega(1)\sum_{i \in [t]} (\norm{x_{t+1} - y_t}/R)^{2/3}}$. By requiring our bisection procedure to find values for which $\norm{x_{t+1} - y_t} \in [r/2,r)$, we obtain the claimed $\Otil{(R/r)^{2/3}}$ complexity bound.  Our analysis of our method closely follows the previous ball-oracle acceleration proof of \cite{carmon2020acceleration}, which itself draws from prior analyses of Monteiro-Svaiter-type algorithms~\cite{gasnikov19near,bubeck2019complexity}.

The differences between our algorithm and proof and those in \cite{carmon2020acceleration} center around handling non-smoothness and less accurate ball oracle outputs. In particular, in \Cref{line:v-update} of \Cref{alg:ms-bacon-redux} we estimate $\grad f(\prox[\lambda_{t+1}]{f}(y_t))$ as $\lambda_{t+1}(y_t - x_{t+1})$, which allows us to avoid smoothness assumptions but requires a somewhat different proof of the main potential bound. We also remove the assumption that $x_0$ is within distance $R$ from a global minimizer of $f$, and instead compare the function value of our output to the minimizer of $f$ in a ball of radius $R$. Our bisection subroutine and its analysis also differ from its counterpart in \cite{carmon2020acceleration}; we explain these differences in the next subsection. 

Our analysis of the \Cref{alg:ms-bacon-redux} outer loop relies on the following guarantee for our bisection subroutine, which we prove in \Cref{sec:ms-bacon-redux-proofs-bisection}.

\begin{restatable}[Bisection]{prop}{propbisectiongen}\label{prop:bisection-gen}
Let $f : \R^d \rightarrow \R$ be $\Lf$-Lipschitz and convex, and let $x,v\in\R^d$, $\eps, r, R \in \R_{> 0}$ satisfy $\eps \leq \Lf R$, $r \leq R$ and $\norm{x-v}\le 2R$. Given $\lambda_{\max}\ge \tfrac{2\Lf}{r}$ and $\lambda_{\min}\in(0,\lambda_{\max})$, $\linesearch(x,v,A)$ outputs $\lambda \in [\lambda_{\min}, \lambda_{\max}]$ such that 
\[\prox{f}(y_\lambda) = \bprox{f}(y_\lambda)\]
(i.e., $\norm{\prox{f}(y_\lambda)-y_\lambda}\le r$). The subroutine uses $O(\log(\tfrac{\lambda_{\max}}{\lambda_{\min}})+\log (\tfrac{R + \Lf/\lambda_{\min}}{r}))$ calls to $\oracle[\lambda',\frac{r}{17}]{\cdot}$ with $\lambda' \in [\half\lambda, \lambda_{\max}]$. Moreover, for  $\alpha_{2\lambda A} = \frac{2\lambda A}{1+2\lambda A + \sqrt{1+4\lambda A}}$ and  $y_{\lambda} \defeq \alpha_{2 \lambda A } x + (1-\alpha_{2 \lambda A}) v$ one of the following outcomes must occur:
	\begin{enumerate}[label=(\alph*)]
		\item \label{item:outcome-large-move} $ 
		\lambda \in [2\lambda_{\min}, \lambda_{\max}]$ and  $\norm{\prox{f}(y_\lambda)-y_\lambda} > \frac{3r}{4}$, or
		\item \label{item:outcome-small-lambda} $\lambda < 2\lambda_{\min}$.
	\end{enumerate}
	When taking $\lambda_{\max}=\tfrac{2\Lf}{\eps}$, $\lambda_{\min} = \tfrac{\eps}{6rR}$, the number of calls to $\oracle[\lambda',\frac{r}{17}]{\cdot}$ is bounded by $O(\log\tfrac{\Lf R^2}{r\eps})$.
\end{restatable}

For the remainder of this section, we fix a parameter $R$ and let
\begin{equation*}
	\xopt \in \argmin_{x \in \ball_R(x_0)} f(x)
\end{equation*}
denote a minimizer of $f$ in a ball of radius $R$ around the initial point $x_0$. (If it is not unique, we choose one arbitrarily). Based on the iterates $\{x_t,v_t,A_t\}$ generated by \Cref{alg:ms-bacon-redux}, we define the following quantities:
\begin{equation}\label{eq:potential-def}
	E_t \defeq f(x_t)-f(\xopt),~\hat{E}_t = E_t - \frac{\epsilon}{4},~
	D_t = \half \norm{v_t-\xopt}^2,~\mbox{and}~
	P_t = A_t \hat{E}_t + D_t.
\end{equation}

In the following lemma we prove our main potential decrease bound, under the conditions that the iterates $x_t,v_t$ are within distance $2R$ of each other and (implicitly) that a suitably good solution has yet to be found. We establish these conditions inductively in subsequent lemmas and leverage this to lower bound the growth of $A_t$ and prove \Cref{thm:ms-bacon-redux}.
\begin{lem}
	\label{lem:inductive-potential}
	Let $f$ be a convex function, $x_0 \in \R^d$ and $\epsilon, R>0$. If at iteration $t$ 
	 of \Cref{alg:ms-bacon-redux} the following conditions hold,
	\begin{enumerate}[label=(\alph*)]
		\item \label{item:cond-dist-bound} $\norm{x_t - v_t} \leq 2R$
		\item \label{item:cond-lower-bound} $\lambda_{t+1} \ge  \frac{\eps}{3rR}$ 
	\end{enumerate}
	we have 
	\[
	P_{t+1} - P_t
	\leq -\frac{A_{t+1} \lambda_{t+1} r^2}{12}. 
	\]
\end{lem}
\begin{proof}
	Let 
	\[
	\xhat_{t+1} \defeq \bprox[r,\lambda_{t+1}]{f}(y_t) = \argmin_{x \in \ball_{r}(y_t)} \crl*{
		f(x) + \frac{\lambda_{t+1}}{2} \norm{x-y_t}^2}.
	\]
	 Condition~\ref{item:cond-dist-bound} and 
	\Cref{prop:bisection-gen} 
	guarantee that the ball constraint is inactive, i.e.,
	\begin{equation*}
	\xhat_{t+1} = \bprox[r,\lambda_{t+1}]{f}(y_t)  = \prox[\lambda_{t+1}]{f}(y_t) = \argmin_{x \in \R^d} \crl*{
		f(x) + \frac{\lambda_{t+1}}{2} \norm{x-y_t}^2}.
	\end{equation*}
	(Note that $y_t$ is precisely $y_{\lambda}$ defined in \Cref{prop:bisection-gen}). Consequently, by \Cref{lem:grad_mapping} we have that $g_{t + 1} \defeq \lambda_{t+1} (y_t - \xhat_{t+1})\in \partial f(\hat{x}_{t+1})$, i.e.,
	\begin{align}
		f(u) &\geq f(\xhat_{t+1}) + g_{t + 1}^\top (u - \xhat_{t+1}) 
		\text{ for all } u \in \R^d.
		\label{eq:g_is_grad}
	\end{align}
	 Further, since 
	 $
	 v_{t+1}= \argmin_{v \in \ball_R(x_0)}\left\{a_{t+1} \left\langle y_{t}-x_{t+1},v \right\rangle +\frac{1}{2}\norm{ v-v_{t}} ^{2}\right\}
	 $
	 and $\xopt \in \ball_R(x_0)$ applying \Cref{lem:grad_mapping} again (with $u = \xopt\in \ball_R(x_0)$) yields that
	\begin{align}
		a_{t+1} \lambda_{t+1} \left\langle y_{t}-x_{t+1},v_{t+1 } -\xopt\right\rangle 
		&\le\frac{1}{2}\norm{ v_{t} - \xopt} ^{2}-\frac{1}{2}\norm{ v_{t+1} - \xopt} ^{2}-\frac{1}{2}\norm{ v_{t+1}-v_{t}} ^{2}
		\nonumber\\
		&= D_t - D_{t +1} - \frac{1}{2} \norm{v_{t+1} - v_t}^2
		\,.
		\label{eq:exact-prox-v}
	\end{align}

	Our proof strategy is to upper and lower bound the inner product $\inner{g_{t+1}}{v_{t+1}-\xopt}$. In particular, we will lower bound this inner product using \eqref{eq:g_is_grad} and  upper bound it using \eqref{eq:exact-prox-v}. Towards this end, we define the point
	\[
	\q_t \defeq \frac{A_t}{A_{t+1}} x_t + \frac{a_{t+1}}{A_{t+1}} v_{t+1}.
	\]
	We remark that the use of $\q_t$ is inspired from the acceleration analysis of~\citet{allen2017linear}. From the definition of $y_t$, we obtain
	\begin{equation}\label{eq:coupling-relation}
		v_{t}=\frac{1}{a_{t+1}}\left(A_{t+1}y_{t}-A_{t}x_{t}\right) \quad \text{ and } \quad v_{t}-v_{t+1}=\frac{A_{t+1}}{a_{t+1}}\left(y_{t}-\q_{t}\right).
	\end{equation}
	Recalling that $A_{t+1}=A_t + a_{t+1}$, we have
	\begin{equation}
		\label{eqn:vt}
		v_{t+1}=\frac{1}{a_{t+1}}\left(A_{t+1}\q_{t}-A_{t}x_{t}\right)=\hx_{t+1}+\frac{A_{t}}{a_{t+1}}\left(\hx_{t+1}-x_{t}\right)-\frac{A_{t+1}}{a_{t+1}}\left(\hx_{t+1}-\q_{t}\right).
	\end{equation}

	To begin our inner product lower bound, we note that
	\begin{align*}
		\left\langle g_{t + 1},v_{t+1}-\xopt\right\rangle   
		&=\left\langle g_{t + 1} , \hx_{t+1}-\xopt\right\rangle +\frac{A_{t}}{a_{t+1}}\left\langle g_{t + 1} , \hx_{t+1}-x_{t}\right\rangle -\frac{A_{t+1}}{a_{t+1}}\left\langle g_{t + 1} , \hx_{t+1}-\q_{t}\right\rangle \\
		& \ge f\left(\hx_{t+1}\right)-f\left(\xopt\right)+\frac{A_{t}}{a_{t+1}}\left[f\left(\hx_{t+1}\right)-f\left(x_{t}\right)\right]-\frac{A_{t+1}}{a_{t+1}}\left\langle g_{t + 1} , \hx_{t+1}-\q_{t}\right\rangle \\
		& =
			\frac{A_{t+1}}{a_{t+1}}\left[f\left(\hx_{t+1}\right)-f\left(\xopt\right)\right]
			-\frac{A_{t}}{a_{t+1}} \left[f\left(x_{t}\right)-f\left(\xopt\right)\right]
			-\frac{A_{t+1}}{a_{t+1}}\left\langle g_{t + 1} , \hx_{t+1}-\q_{t}\right\rangle,
		\label{eq:bacon-intermediate-convexity-bound}
		\numberthis
	\end{align*}
	where the inequality follows from \eqref{eq:g_is_grad}. To relate $f(\hx_{t+1})$ to $f(x_{t+1})$ in \eqref{eq:bacon-intermediate-convexity-bound}, we use the approximation guarantee~\eqref{eq:broo-req} defining $\oracle{\cdot}$ to obtain
	\begin{align*}
		f(\xhat_{t+1})  &\ge 
		f(x_{t+1}) + \frac{\lambda_{t+1}}{2} \norm{x_{t+1} - y_t}^2 - \frac{\lambda_{t+1}}{2} \norm{\xhat_{t+1} - y_t}^2 - \frac{\lambda_{t+1}\delta_{t+1}^2}{2}
		\nonumber
		\\
		&\ge
		f(x_{t+1}) + \frac{\lambda_{t+1}}{2} \norm{x_{t+1} - y_t}^2 
		+ \left\langle g_{t +1} , \hx_{t+1} - \q_t\right\rangle 
		- \frac{\lambda_{t+1}}{2} \norm{y_{t} - \q_t} ^{2}
		 - \frac{\lambda_{t+1}\delta_{t+1}^2}{2}
		\label{eq:boo-in-bacon-subopt}
	\end{align*}
	where we used $\left\langle g_{t +1} , \hx_{t+1} - \q_t\right\rangle 
	= \frac{\lambda_{t+1}}{2} \norm{y_{t} - \q_t} ^{2}-\frac{\lambda_{t+1}}{2}\norm{ \hx_{t+1} - \q_t} ^{2}-\frac{\lambda_{t+1}}{2}\norm{ \hx_{t+1}-y_{t}} ^{2}$ (as in \Cref{lem:grad_mapping}) to obtain the second inequality. Substituting into~\eqref{eq:bacon-intermediate-convexity-bound} and recalling that $\E_t = f(x_t) - f(\xopt)$ yields
	\begin{flalign*}
		\inner{ g_{t + 1} }{v_{t+1}-\xopt} &\ge \frac{A_{t+1}}{a_{t+1}} E_{t+1} - \frac{A_{t}}{a_{t+1}} E_{t} 
		+ \frac{A_{t+1}\lambda_{t+1}}{2a_{t+1}}\brk*{\norm{x_{t+1} - y_t}^2 - \delta_{t+1}^2 -  \norm{y_t - \q_t}^2 }.
	\end{flalign*}
	The lower bound $\lambda_{t+1} \ge \frac{\epsilon}{3rR}$ implies that $\delta_{t+1}=\frac{\epsilon}{12\lambda_{t+1}R} \le \frac{r}{4}$. 
	Moreover, by condition \ref{item:cond-lower-bound} ($\lambda_{t+1} \ge \frac{\epsilon}{3rR}$) and \Cref{prop:bisection-gen} we have that $\norm{\hx_{t+1}-y_t} \ge 3r/4$. Applying 
	 \Cref{lem:prox_oracle_dist} we conclude that
	 \begin{equation*}
	 	\norm{x_{t+1} - y_t}^2 - \delta_{t+1}^2 \ge (\norm{\hx_{t+1}-y_t}-\delta_{t+1})^2 - \delta_{t+1}^2 \ge r^2 \brk*{ (\tfrac{3}{4}-\tfrac{1}{4})^2 - (\tfrac{1}{4})^2} \ge \frac{r^2}{6}.
	 \end{equation*}
 	Substituting back, we have
 	\begin{equation}\label{eq:bacon-inner-prod-lb}
 			\inner{g_{t + 1} }{v_{t+1}-\xopt} \ge \frac{A_{t+1}}{a_{t+1}} E_{t+1} + \frac{A_{t}}{a_{t+1}} E_{t} - \frac{A_{t+1}\lambda_{t+1}r^2}{12 a_{t+1}} - \frac{A_{t+1}\lambda_{t+1}}{2 a_{t+1}}\norm{y_t - \q_t}^2.
 	\end{equation}
	
	We now proceed to upper bound $\inner{ g_{t + 1} }{v_{t+1}-\xopt}$. Recalling $g_{t + 1} = \lambda_{t+1} (y_t - \xhat_{t+1})$ we obtain
	\begin{flalign*}
		\left\langle g_{t+1} , v_{t+1}-\xopt\right\rangle &=\lambda_{t+1} \left\langle y_{t}-x_{t+1},v_{t+1}-\xopt\right\rangle +\lambda_{t+1}\left\langle x_{t+1}-\hx_{t+1},v_{t+1}-\xopt\right\rangle \\ & \le\lambda_{t+1}\left\langle y_{t}-x_{t+1},v_{t+1}-\xopt\right\rangle +\lambda_{t+1} \norm{x_{t+1} - \hx_{t+1}} \norm{v_{t+1}-\xopt}.
		\label{eq:bacon-upper-intermediate}\numberthis
	\end{flalign*}
	To bound the term $\lambda_{t+1} \norm{x_{t+1} - \hx_{t+1}} \norm{v_{t+1}-\xopt}$, note that $\norm{x_{t+1} - \hx_{t+1}}\le\delta_{t+1}$ by \Cref{lem:prox_oracle_dist}, $\norm{v_{t+1}-\xopt}\le 2R$ since $v_{t+1}$ and $\xopt$ are both in $\ball_R(x_0)$ by assumption, and $\lambda_{t+1}\delta_{t+1} \cdot 2R = \frac{\epsilon}{6} \le \frac{\epsilon}{4}$. To bound the term $\lambda_{t+1}\inner{y_t-x_{t+1}}{v_{t+1}-\xopt}$ we apply \eqref{eq:exact-prox-v}. Applying these bounds to \eqref{eq:bacon-upper-intermediate} yields 
	\begin{flalign*}
		\inner{g_{t+1}}{v_{t+1}-\xopt}& \le \frac{1}{a_{t+1}}\brk*{D_t - D_{t+1} - \half \norm{v_{t+1}-v_{t}}^2 +\frac{ a_{t+1}\epsilon}{4}}
		\\&
		=
		\frac{1}{a_{t+1}}\brk*{D_t - D_{t+1} - \frac{A_{t+1}^2}{2a_{t+1}^2}\norm{y_t-\q_t}^2 +(A_{t+1}-A_{t})\frac{ \epsilon}{4}},
		\label{eq:bacon-inner-prod-ub} \numberthis
	\end{flalign*}
	where the equality is due to~\eqref{eq:coupling-relation} and $a_{t+1}=A_{t+1}-A_{t}$.
	
	Combining the lower and upper bound~\eqref{eq:bacon-inner-prod-lb} and~\eqref{eq:bacon-inner-prod-ub} and rearranging, we have
	\begin{equation*}
		A_{t+1}(E_{t+1}-\tfrac{\epsilon}{4})+D_{t+1}-A_{t}(E_{t+1}-\tfrac{\epsilon}{4})-D_{t} \le  -\frac{A_{t+1} \lambda_{t+1} r^2}{12} + \left(A_{t+1}\lambda_{t+1}-\frac{A_{t+1}^{2}}{a_{t+1}^{2}}\right)\half\norm{ y_{t}-\q_{t}} ^{2}.
	\end{equation*}
	The proof is complete upon noticing that $A_{t+1}\lambda_{t+1}-\frac{A_{t+1}^{2}}{a_{t+1}^{2}}=0$ (since $A_{t+1}=a_{t+1}^2\lambda_{t+1}$) and that $P_t = A_t \hat{E}_t + D_t = A_t (E_t - \frac{\eps}{4}) + D_t$.
\end{proof}

\Cref{lem:inductive-potential} shows that the potential $P_t$ decreases significantly whenever $\lambda_{t+1}$ is not too small and $\norm{x_t-v_t}\le 2R$ holds; we now the latter condition inductively.

\begin{lem}\label{lem:full-potential-bound}
	Fix $t\ge 1$. In  \Cref{alg:ms-bacon-redux} if  $f(x_i) - f(\xopt) > \epsilon$ and $\lambda_{i} \ge \frac{\epsilon}{3rR}$ for all $0 \leq i \le t$ then
	\[
	\norm{x_t - \xopt} \leq R
	\mbox{,}~~\norm{v_t - \xopt} \leq R,~~\mbox{and}~~P_{t} - P_0 \leq -\sum_{i \in [t-1]} \frac{A_{i+1} \lambda_{i+1} r^2 }{12}.
	\]
\end{lem}
\begin{proof}
	We proceed by induction on $t$. For the base case of $t=0$, we note that $\norm{v_0-\xopt}=\norm{x_0 - \xopt} \leq R$ by assumption, and that $P_0  - P_0 \le 0$ trivially. 
	Therefore we assume the inductive hypothesis that the lemma statement holds for iteration $t$, and show that it also holds for $t+1$. 
	First, we note that $\norm{x_t - \xopt} \leq R$, $\norm{v_t - \xopt} \leq R$ and $\lambda_{t+1} \ge \frac{\epsilon}{3rR}$ satisfy the conditions of \Cref{lem:inductive-potential} and consequently $P_{t+1} \le P_{t} - \frac{1}{12}A_{t+1}\lambda_{t+1}r^2$; together with the inductive hypothesis this establishes
	\[
	P_{t+1} - P_0 \leq -\sum_{i=0}^{t} \frac{A_{i+1} r^2 \lambda_{i+1}}{12}.
	\]
	Next, we note that $f(x_{t+1})-f(\xopt)>\epsilon$ implies that $\hat{E}_{t+1} = f(x_{t+1})-f(\xopt) - \frac{\eps}{4} > 0$.
	Recalling the definition~\eqref{eq:potential-def} and $A_0=0$, this implies
	\begin{equation*}
		\half\norm{v_t - \xopt}^2 = D_t \le A_t \hat{E}_t + D_t = P_t \le P_0 = D_0 = \half\norm{x_0- \xopt}^2 \le \half R^2,
	\end{equation*}
	and consequently 
	\[\norm{v_{t+1}-\xopt} \le R.\]
	
	To complete the induction step we need to argue that  $\norm{x_{t+1}-\xopt} \le R$. To that end, we invoke \Cref{lem:dist_monotone} (recalling that $f(x_{t+1}) - f(\xopt) > \epsilon$ by assumption) which shows that $\norm{x_{t+1}-\xopt} \leq \norm{y_{t}-\xopt}$ or $f(x_{t+1})-f(\xopt) \le \frac{\epsilon}{2}$. The definition $y_t = \frac{A_t}{A_{t+1}}x_t + \frac{a_{t+1}}{A_{t+1}} v_t$ gives
	\begin{equation*}
			\norm{x_{t+1}-\xopt}\le \norm{y_{t}-\xopt}\le \frac{A_t}{A_{t+1}} \norm{x_t - \xopt} + \frac{a_{t+1}}{A_{t+1}} \norm{v_t - \xopt} \leq R,
	\end{equation*}
	where the final bounds holds since $A_{t+1}=A_t + a_{t+1}$ and  $\norm{x_t - \xopt} , \norm{v_t - \xopt}\le R$ by the inductive assumption.
\end{proof}

\begin{lem}\label{lem:At_increase}
Fix $t\ge 1$. In  \Cref{alg:ms-bacon-redux} if  $f(x_i) - f(\xopt) > \epsilon$ and $\lambda_{i} \ge \frac{\epsilon}{3rR}$ for all $0 \leq i \le t$ then
\[
A_{t} \geq \exp\left(\frac{r^{2/3}}{R^{2/3}}(t-1)\right) A_1\,.
\]
\end{lem}
\begin{proof}
\Cref{lem:full-potential-bound} implies that
\[
P_{t} - P_0 \leq -\sum_{i=0}^{t} \frac{A_{i+1} r^2 \lambda_{i+1}}{12}\,.
\]
Further, that $f(x_i)-f(\xopt) \ge  \epsilon$ for all $i \le t$ implies $\hat{E}_t \ge 0$ and therefore $P_t \ge 0$. Combining these facts with the facts that $A_0=0$ and the $A_i$ increase monotonically yields
\begin{equation}\label{eq:potential-bound-rearranged}
	\sum_{i=0}^{t-1} A_{i+1}\lambda_{i+1} \le \frac{12}{r^2} P_0 = \frac{12}{r^2} D_0 \le \frac{6R^2}{r^2}.
\end{equation}

Next, note that the reverse H\"older inequality with $p= 2/3$ states that for any $u,v \in \R^d_{> 0}$
\[
\left\langle u, v \right\rangle \geq \left( \sum_{i \in [d]} u_i^{2/3} \right)^{3/2} \cdot \left( \sum_{i \in [d]} v_i^{-2} \right)^{-1/2}.
\]
We therefore have
\begin{flalign*}
	\sqrt{A_t} \overge{(i)} \frac{1}{2} \sum_{i \in [t]} \frac{1}{\sqrt{\lambda_i}} 
	&\overge{(ii)} \frac{1}{2} \left( \sum_{i \in [t]} \left(\sqrt{A_i}\right)^{2/3} \right)^{3/2} \cdot \left( \sum_{i \in[t]} \left( \frac{1}{\sqrt{A_i \lambda_i}} \right)^{-2} \right)^{-1/2} \\& \overge{(iii)} 
	\frac{1}{2} \left( \sum_{i \in [t]} \left(\sqrt{A_i}\right)^{2/3} \right)^{3/2}\cdot \frac{r}{\sqrt{6}R},
\end{flalign*}
where we used $(i)$ \Cref{lemma:ak_bounds}, $(ii)$ the reverse H\"older inequality with $u_i=\sqrt{A_i}$ and $v_i=1/\sqrt{A_i\lambda_i}$, and $(iii)$ the bound \eqref{eq:potential-bound-rearranged}. Rearranging, we have
\begin{equation}
\label{eq:At_bound_1}
A_t^{1/3} \geq  \frac{r^{2/3}}{3 R^{2/3}}  \left( \sum_{i \in [t]} A_i^{1/3} \right).
\end{equation}
Lemma 28 of \cite{carmon2020acceleration} shows that for any sequence of $A_t$ that satisfy \eqref{eq:At_bound_1} for all $t \in [T]$ also satisfies
\[
A_T^{1/3} \geq 
\exp\left( \frac{ r^{2/3}}{3 R^{2/3}} (T - 1) \right) A_1^{1/3} 
\] 
and the result follows.
\end{proof}

We are now ready to prove our main theorem.

\newcommand{\Tout}{T_o}

\begin{proof}[Proof of \Cref{thm:ms-bacon-redux}]
	This proof proceeds in parts. First, we show that whenever the algorithm terminates, i.e. one of the conditions of \Cref{line:outer_check} is met, then \Cref{alg:ms-bacon-redux} outputs a point $\xret$ with $f(\xret)-f(\xopt)\le \epsilon$ on \Cref{line:outer_return} . Next, we bound $\Tout$, the value of $t$ at termination on \Cref{line:outer_return}, and show that $\Tout = O\prn*{
			\prn*{ \frac{R}{r}}^{2/3}
			 \log \prn*{\frac{[f(x_0 - f(\xopt))] \cdot R^2}{\epsilon \cdot r^2}}
		}$. Leveraging these facts we complete the proof.

	\paragraph{Termination due to small $\lambda_{\Tout+1}$.}
		Consider the case where the algorithm terminates because $\lambda_{\Tout+1} <  \frac{\epsilon}{3rR}$ on \Cref{line:outer_check}. Further, suppose that $f(x_i) - f(\xopt) > \epsilon$ for all $0 \leq i \leq \Tout$ as otherwise $f(\xret) - f(\xopt) \leq \epsilon$ as desired. 
		 By definition of $\Tout$, we have $\lambda_{t} \ge \frac{\epsilon}{3rR}$ for all $t\le \Tout$ (otherwise we would have terminated earlier). Applying \Cref{lem:full-potential-bound}, we conclude it must be that $\norm{x_{\Tout}-\xopt}\le R$ and $\norm{v_{\Tout}-\xopt}\le R$ and therefore  $\norm{y_{\Tout}-\xopt}\le R$. By \Cref{prop:bisection-gen} we have that $\hx_{\Tout+1}=\prox{f}(y_{\Tout})=\bprox{f}(y_{\Tout})$ and thus $\norm{\hx_{\Tout+1}-y_{\Tout}}<r$. 
	\Cref{lem:prox_props} with $x_0 = y_{\Tout}$, $x_\lambda=\hx_{\Tout+1}$,  $y=\xopt$ and $\Theta=r$ then yields
	\begin{equation*}
		f(\hx_{\Tout+1}) - f(\xopt) \le \lambda_{t+1} r R < \frac{\epsilon}{3}.
	\end{equation*}
	Moreover, the \BOO guarantee~\eqref{eq:broo-req} gives
	\begin{equation*}
		f(x_{\Tout+1}) \le f(\hx_{\Tout+1}) + \frac{\lambda_{\Tout+1}}{2}\brk*{\norm{\hx_{\Tout+1}-y_t}^2 - \norm{x_{\Tout+1}-y_t}^2 + \delta_{\Tout+1}^2} \le \frac{\lambda_{\Tout+1}}{2}(r^2 + \delta_{\Tout+1}^2).
	\end{equation*}
	Noting that $\lambda_{\Tout+1} \in [\frac{\epsilon}{6rR}, \frac{\epsilon}{3rR}]$  and that $\delta_{\Tout+1}=\frac{\epsilon}{12\lambda_{\Tout+1}R}$, we have have $\lambda_{\Tout+1}r^2 \le \epsilon \cdot \frac{r}{3R}\le \epsilon/3$ and $\lambda_{\Tout+1} \delta_{\Tout+1}^2 \le \epsilon \cdot \frac{\epsilon}{12R} \cdot \frac{r}{2} \le \epsilon/24$. Substituting back, we have
	\begin{equation*}
		f(x_{\Tout+1})-f(\xopt) \le f(\hx_{\Tout+1})-f(\xopt) + \frac{9\epsilon}{49} \le \frac{\epsilon}{3} +  \frac{9\epsilon}{49}  \le \epsilon.
	\end{equation*}

	\paragraph{Termination due to large $A_{\Tout+1}$.}
	Next consider the case where $\lambda_{t+1} \ge  \frac{\epsilon}{3rR}$ for all $t\le \Tout$ but $A_{\Tout+1} \ge \frac{R^2}{\epsilon}$. In this case,  \Cref{lem:full-potential-bound} implies that unless $f(\xret) - f(\xopt) \leq \epsilon$
	\begin{equation*}
		A_{\Tout+1} \prn*{ f(x_{\Tout+1})-f(\xopt) - \frac{\eps}{4}}
		=A_{\Tout+1} \hat{E}_{\Tout+1} \le P_{\Tout+1} \le P_0 = D_0 = \frac{1}{2}\norm{x_0 - \xopt}^2 \le \frac{R^2}{2}.
	\end{equation*}
	Dividing by $A_{\Tout+1} \ge \frac{R^2}{\epsilon}$, we obtain
	\begin{equation*}
		f(x_{\Tout+1})-f(\xopt) \le \frac{\eps}{4} + \frac{\epsilon}{2} \le \epsilon.
	\end{equation*}
 	
 	\paragraph{Termination due to distant $x_{\Tout +1}, v_{\Tout + 1}$.}
 	Next consider the case where  $\lambda_{t+1} \ge  \frac{\epsilon}{3rR}$ for all $t\le \Tout$ but $\norm{x_{\Tout +1} - v_{\Tout + 1}} > 2R$. This implies that either $\norm{x_{\Tout +1} - \xopt} > R$ or $\norm{v_{\Tout + 1} - \xopt} > R$. Consequently, \Cref{lem:full-potential-bound} implies that $f(x_i) - f(\xopt) \leq \epsilon$ for some $0 \leq i \leq \Tout + 1$. 
 	
 	\paragraph{Termination due to slow $A_t$ growth.}
 	Next consider the case where  $\lambda_{t+1} \ge  \frac{\epsilon}{3rR}$ for all $t\le \Tout$ but $A_{\Tout + 1} < \exp\left(\frac{r^{2/3}}{R^{2/3}}(\Tout-1)\right) A_1$. In this case \Cref{lem:At_increase} implies that $f(x_i) - f(\xopt) \leq \epsilon$ for some $0 \leq i \leq \Tout + 1$ and therefore $f(\xret) - f(\xopt) \leq \epsilon$ as desired. 
 	
 	\vspace{\baselineskip}
 	
 	We conclude that in each of the four possible causes for the algorithm to terminate, i.e. one of the conditions in \Cref{line:outer_check} to be true, the algorithm $\xret$ such that $f(\xret)-f(\xopt)\le \epsilon$ as desired.

 	\paragraph{Iteration Bound.} 
 	We now show that $\Tout = O\left(\frac{ R^{2/3}}{r^{2/3}} \log \left( \frac{[f(x_0) - f(\xopt)] R}{ \epsilon r} \right)\right)$.
 	Note that by definition of $\Tout$ as the iterate index $t$ for which termination on \Cref{line:outer_return} occurs, we have
 	\begin{equation*}
 		 \exp\left(\frac{r^{2/3}}{R^{2/3}} \cdot (\Tout-1) \right) A_1 \le A_{\Tout} \le \frac{R^2}{\epsilon}
 	\end{equation*}
 	since otherwise we would have terminated on at $t=\Tout-1$. Consequently we have that 
 	\begin{equation}
 		\label{eq:Tout-in-terms-of-A1}
 		\Tout=O\left(\frac{ R^{2/3}}{r^{2/3}} \log \left( \frac{ R^2}{ A_1\epsilon } \right)\right),
 	\end{equation}
 	 and it remains to lower bound $A_1 = a_1 = \frac{1}{2 \lambda_1}$ (since $A_0 = 0$).
 	
 	Note that since $x_0 = v_0$ it is the case that $y_0 = x_0$. Consequently, if  $\Tout > 0$, \Cref{line:outer_check} implies that $\lambda_1 \geq \frac{\epsilon}{3 r R}$, i.e., \cref{item:outcome-small-lambda} of \Cref{prop:bisection-gen} does not hold). Consequently, \Cref{prop:bisection-gen} implies that $\norm{\prox[\lambda_{1}]{f}(x_0) - x_{0}} \in [\frac{3r}{4}, r]$ and therefore
 	\[
 	f(x_0) \geq f\left(\prox[\lambda_{1}]{f}(x_0) \right) + \frac{\lambda}{2} \norm{\prox[\lambda_{1}]{f}(x_0) - x_0}^2
 	\geq f\left(\prox[\lambda_{1}]{f}(x_0) \right) + \frac{9 \lambda r^2}{32}.
 	\] 
 	Further, since $\norm{\prox[\lambda_{1}]{f}(x_0) - x_0} \leq r  \leq R$ we know that 
 	$f(\prox[\lambda_{1}]{f}(x_0))  \geq f(\xopt)$. Consequently, 
 	\begin{equation*}
 		\lambda_1 \leq \frac{32 [f(x_0) - f(\xopt)]}{9 r^2}
 		\text{ and }
 		A_1 =\frac{1}{\lambda_1} \ge \frac{9 r^2}{32 [f(x_0) - f(\xopt)]}.
 	\end{equation*}
 	Substituting back to~\eqref{eq:Tout-in-terms-of-A1} yields the claimed bound on $\Tout$.

	\paragraph{Remaining guarantees.} 
	To complete the proof, we observe that in each of the algorithm's $O\left( \left(\frac{R}{r}\right)^{2/3} \log\left( \frac{\Lf R^2 }{r\eps} \right) \right)$ iterations, \Cref{prop:bisection-gen} guarantees that we perform at most $O(\log \frac{\Lf R^2}{r\eps})$ queries to $\oracle{\cdot}$. It also guarantees that the queried $\lambda$ satisfies $\lambda \in [\frac{\epsilon}{24rR}, \frac{4\Lf}{r}]$, yielding \cref{item:lambda-bounds} in \Cref{thm:ub}. Finally, for the sequence $\lambda\fancyind{1}, \ldots, \lambda\fancyind{T}$ of $\lambda$ values queried during the execution of \Cref{alg:ms-bacon-redux}, we have
	\begin{flalign*}
			\sum_{i \in [T]} \frac{1}{\sqrt{\lambda_{(i)}}} = \sum_{t \in [\Tout+1]} \sum_{\substack{i  \text{ seen in line search}\\
				\text{for iteration $t$} }} \frac{1}{\sqrt{\lambda_{(i)}}} 
			&\overle{(i)} \sum_{i \in [\Tout + 1]}  \frac{O(1)}{\sqrt{\lambda_{t}}} \log\left( \frac{\Lf R^2}{r\eps}\right)
			\\& \overle{(ii)}
			 O\prn*{\sqrt{A_{\Tout}}+\frac{1}{\sqrt{\lambda_{\min}}} }
	\log\left( \frac{\Lf R^2}{r\eps}\right),
	\end{flalign*}
	due to $(i)$ \Cref{prop:bisection-gen} and $(ii)$ \Cref{lemma:ak_bounds} and $\lambda_{\Tout+1} \ge \lambda_{\min} = \frac{\epsilon}{6rR}$ for all $t$. Finally, \Cref{line:outer_check} guarantees $A_{\Tout} < \frac{R^2}{\epsilon}$, and (since $r \le R$)  $\sqrt{A_{\Tout}}+\frac{1}{\sqrt{\lambda_{\min}}} = O(\frac{R}{\sqrt{\epsilon}})$, giving \cref{item:lambda-sum-bound} in  \Cref{thm:ub}.
\end{proof}

\subsection{Bisection analysis}
\label{sec:ms-bacon-redux-proofs-bisection}
In this section we prove \Cref{prop:bisection-gen}, a guarantee on our bisection subroutine,  restated below.

\newcommand{\hxl}[1][\lambda]{\hat{x}_{#1}}
\newcommand{\yl}[1][\lambda]{y_{#1}}
\newcommand{\hDl}[1][\lambda]{\hat{\Delta}(#1)}
\newcommand{\Dl}[1][\lambda]{\Delta(#1)}
\newcommand{\dl}[1][\lambda]{\delta_{#1}}

\propbisectiongen*

\subsubsection{Preliminaries and discussion}

\paragraph{Notation.}
To analyze the bisection procedure, we use the following functions of $\lambda\ge0$. Fixing $x,v\in\R^d$ and $A\ge 0$, we define
\begin{equation*}
	\yl  \defeq 
	\alpha_{2A\lambda'} \cdot x + (1-\alpha_{2A\lambda'}) \cdot,~~\mbox{where}~~
	\alpha_\tau \defeq \frac{\tau}{1+\tau+\sqrt{1 + 2\tau}}.
\end{equation*}
We also define
\begin{equation*}
	\hxl \defeq \prox{f}(\yl) = \argmin_{x\in\R^d}\crl*{f(x) + \frac{\lambda}{2}\norm{x-\yl}^2}
\end{equation*}
and
\begin{equation}\label{eq:hDl-def}
	\hDl \defeq \norm{\hxl - \yl}.
\end{equation}
We approximate $\hxl$ using the \BOO output $\oracle{\yl}$ with $\delta = r/17$,  and write
\begin{equation*}
	\Dl \defeq \norm{\oracle{\yl}-\yl}
\end{equation*}
for our approximation to $\hDl$. Note that, depending on the \BOO implementation, $\Dl$ need not be deterministic. Our bisection procedure implicitly assumes that the \BOO is called only once per input $x,v,A$ and distinct value of $\lambda$, and that subsequent references to $\Dl$ use cached values of $\oracle{\yl}$.  

\paragraph{Algorithm overview.}
The goal of $\linesearch$ is to find a value of $\lambda$ where 
$\hDl < r$ so that $\hxl=\prox{f}(\yl)=\bprox{f}(\yl)$ is well-approximated by the \BOO output $\oracle{\yl}$. In addition, the bisection has to guarantee that either 
\begin{itemize}
	\item $\hDl \ge 3r/4$, which implies that the outer loop of \Cref{alg:ms-bacon-redux} makes sufficient progress (see \Cref{lem:inductive-potential}), or
	\item  $\lambda <2 \lambda_{\min}$, which implies that $\hxl$ is near-optimal (as we argue in the proof of \Cref{thm:ub}).
\end{itemize}
Our procedure (given in Algorithm~\ref{alg:ms-bacon-redux}) starts with $\lambda = \lambda_{\max}$ sufficiently large to guarantee $\hDl < r$. It then iteratively halves $\lambda$ until finding $\lambda_0$ such that $\Dl[\lambda_0] > \frac{13r}{16}$ or $\lambda_0 < \lambda_{\min}$. In the latter case it returns $2\lambda_0 < 2\lambda_{\min}$ and we note that  $\Dl[2\lambda_0] \le \frac{13r}{16}$ implies $\hDl[2\lambda_0] < r$. In the former case we perform a binary search in the interval $[\lambda_0, 2\lambda_0]$ until we find $\lambda_m$ such that $\Dl[\lambda_m] \in [\frac{13r}{16}, \frac{15r}{16}]$, which implies $\hDl[\lambda_m] \in [3r/4, r)$. 

\paragraph{Comparison to the bisection in \cite{carmon2020acceleration}.}
Our bisection procedure essentially attempts to find $\lambda$ such that $\hDl$ is close to  (but smaller than) $r$. In contrast, \citet{carmon2020acceleration} use the implicit relation $\yl - \hxl = \frac{1}{\lambda} \grad f(\hxl)$ and attempt to find $\lambda$ such that $\norm{\grad f(\hxl)} / \lambda$ is close to $r$. Consequently, the iteration count bounds on the bisection of~\cite{carmon2020acceleration} depend on the continuity $\grad f$, whereas our bisection succeeds even when $f$ is non-smooth and hence $\grad f$ is discontinuous. The key to this improvement is a careful analysis of the continuity of $\hDl$ (see the following subsection). Another novel aspect of our procedure is the two-stage structure where we first iteratively halve $\lambda$ and only then perform a binary search. This structure allows us to guarantee that we never query the \BOO with  $\lambda'$ that is much smaller than the $\lambda$ we eventually output. This guarantee is necessary for proving statement \ref{item:lambda-sum-bound} in \Cref{thm:ub}, which in turn is necessary for establishing an optimal complexity bound in the weakly-smooth regime $\Lg = \Theta(\Lf^2 / \epsilon)$.

\subsubsection{Continuity analysis of $\hDl[\cdot]$}
We begin by proving a bound on the Jacobian of $\dl = \hxl  - \yl$ with respect to $\lambda$ (Lemma~\ref{lem:prox_deriv}) under the assumption that $f$ is twice differentiable. We subsequently remove this assumption via a smoothing argument (Corollary~\ref{cor:prox_deriv}).

\begin{lem}\label{lem:prox_deriv} Let $f : \R^{d}\rightarrow\R$ be convex and twice differentiable, let $\yl\in\R^{d}$ be a differentiable function of $\lambda>0$, $\hxl = \prox{f}(\yl)$, and $\dl = \hxl - \yl$. Then for all $\lambda>0$ we have
	\[
	\normBig{ \frac{d}{d\lambda} \dl }\leq \normBig{\frac{d}{d\lambda}\yl}+\frac{1}{\lambda}\normBig{ \dl }\,.
	\]
\end{lem}

\begin{proof}
	Note that $\grad f(\hxl)+\lambda(\hxl-\yl)=0$.
	Differentiating yields
	\[
	\hessian f(\hxl)\cdot\frac{d}{d\lambda}\hxl+(\hxl - \yl)+\lambda\left(\frac{d}{d\lambda}\hxl-\frac{d}{d\lambda}\yl\right)=0\,.
	\]
	Rearranging the terms, we obtain
	\[
	\left(\hessian f(\hxl) + \lambda \identityMatrix \right)\cdot\left(\frac{d}{d\lambda}\hxl -  \frac{d}{d \lambda} \yl \right) =\yl - \hxl -\hessian f(\hxl) \frac{d}{d \lambda} \yl\,.
	\]	
	Since $f$ is convex, $\hessian f(\hxl)$ is PSD and therefore
	\[
	\left(\frac{d}{d\lambda}\hxl -  \frac{d}{d \lambda} \yl \right) =\left(\hessian f(\hxl) + \lambda \identityMatrix \right)^{-1} \cdot \left(\yl - \hxl -\hessian f(\hxl) \cdot \frac{d}{d \lambda} \yl \right)\,.
	\]
	Note that $\left(\hessian f(\hxl)+\lambda \identityMatrix\right)^{-1}\hessian f(\hxl)$
	is a symmetric PSD matrix with all eigenvalues $\in[0,1]$ and $\left(\hessian f(\hxl)+\lambda \identityMatrix \right)^{-1}$
	is a symmetric PSD matrix with all eigenvalues at most $1/\lambda$.
	Consequently, $\norm{\left(\hessian f(\hxl)+\lambda \identityMatrix \right)^{-1}\hessian f(x)}\leq1$
	and $\norm{\left(\hessian f(\hxl)+\lambda \identityMatrix\right)^{-1}}\leq\frac{1}{\lambda}$
	yielding the claim.
\end{proof}

We now relax the assumption of twice differentiability.

\begin{corollary}\label{cor:prox_deriv} Let $f : \R^{d}\rightarrow\R$ be $\Lf$-Lipschitz and convex, let $\yl\in\R^{d}$ be a differentiable function of $\lambda>0$, $\hxl = \prox{f}(\yl)$, and $\dl = \hxl - \yl$. Then for all $\lambda_1,\lambda_2>0$ we have
	\[
	\normBig{\dl[\lambda_2]-\dl[\lambda_1]}\leq \int_{\lambda=\lambda_1}^{\lambda_2}\normBig{\frac{d}{d\lambda}\yl}d\lambda+\int_{\lambda=\lambda_1}^{\lambda_2}\frac{1}{\lambda}\normBig{ \dl }d\lambda\,.
	\]	
\end{corollary}
\begin{proof}
	For $\sigma > 0$, let $\nu = \mathcal{N}(0;  \sigma^2 I_{d\times d})$ and let $f_{\sigma}(x) = \E f(x+\nu)$. Then $f_\sigma$ is convex, infinitely differentiable and satisfies $0 \le f_\sigma(x)-f(x)\le \Lf \E \norm{\nu} \le  \Lf\sqrt{d}\sigma$ for all $x\in\R^d$.  Thus, applying Lemma~\ref{lem:prox_deriv} for $f_\sigma$ and $\dl^\sigma = \prox{f_\sigma}(\yl)-\yl$, we have
	\[
	\normBig{\dl[\lambda_2]^\sigma-\dl[\lambda_1]^\sigma}\leq \int_{\lambda=\lambda_1}^{\lambda_2}\normBig{\frac{d}{d\lambda}\yl}d\lambda+\int_{\lambda=\lambda_1}^{\lambda_2}\frac{1}{\lambda}\normBig{ \dl^\sigma }d\lambda.
	\]
	
	Noting that $\prox{f_\sigma}(\yl)$ is at most $\Lf \sqrt{d}\sigma$-suboptimal for $f(x) + \frac{\lambda}{2}\norm{x-\yl}^2$, we have 
	\[
	\norm{\dl - \dl^{\sigma}} = \norm{\prox{f}(\yl)-\prox{f_\sigma}(\yl)} \le \sqrt{\frac{2\Lf\sqrt{d}\sigma}{\lambda}}
	\]
	by $\lambda$-strong-convexity of $x\mapsto f(x) + \frac{\lambda}{2}\norm{x-\yl}^2$. 
	Substituting back, we find that
	\[
	\normBig{\dl[\lambda_2]-\dl[\lambda_1]}\leq \int_{\lambda=\lambda_1}^{\lambda_2}\normBig{\frac{d}{d\lambda}\yl}d\lambda+\int_{\lambda=\lambda_1}^{\lambda_2}\frac{1}{\lambda}\normBig{ \dl }d\lambda + \sqrt{\sigma} \cdot \frac{5\lambda_2\Lf^{1/2}d^{1/4}}{\lambda_1^{3/2}}.
	\]
	For all $\sigma >0$. Taking the limit $\sigma\to 0$ concludes the proof.
\end{proof}

With this, we prove our desired bound on the continuity of $\hDl$.

\begin{lem}\label{lem:lambda_change} 
	Let $f:\R^{d}\rightarrow\R$ be convex and $\Lf$-Lipschitz, and for $R>0$
	let $x,v\in\R^d$ such that $\norm{x-v}\le 2R$. For any $0<\lambda_1\le\lambda_2$, the function $\hDl$ defined in eq.~\eqref{eq:hDl-def} satisfies
	\begin{equation*}
		\abs{\hDl[\lambda_1]-\hDl[\lambda_2]} \le \prn*{ R + \frac{\Lf}{\lambda_1}} \log \frac{\lambda_2}{\lambda_1}.
	\end{equation*}

\end{lem}

\begin{proof}
	Note that for all $t > 0$
	\begin{align*}
		\frac{d}{dt} \alpha_t
		&= \frac{1}{1 + t + \sqrt{1 + 2t}}  - 
		\frac{t + t(1 + 2t)^{-1/2}}{(1 + t + \sqrt{1 + 2t})^2}
		=
		\frac{
			(1 + t)\sqrt{1 + 2t} + 1 + 2t - t\sqrt{1 + 2t} - t 
		}{(1 + t + \sqrt{1 +2 t})^2 \sqrt{1 + 2t }}
		\\
		&=  \frac{1}{(1 + t + \sqrt{1 + 2t}) \sqrt{1 + 2t}}
		\in \left[0 , \frac{1}{2t} \right]
	\end{align*}
	Consequently,
	\[
	\abs{\hDl[\lambda_1]-\hDl[\lambda_2]}  \le \normBig{\frac{d}{d\lambda} \yl}
	= \normBig{ (v - x) \frac{d}{d\lambda} \alpha_{2 A \lambda}}
	\leq \frac{2 A}{4 A \lambda} \norm{v - x}  = \frac{1}{2\lambda}\norm{v-x} \le \frac{R}{\lambda}.
	\]
	By \Cref{lem:prox_deriv} we have
	\begin{align*}
		\normBig{\dl[\lambda_2]-\dl[\lambda_1]}\leq \int_{\lambda = \lambda_1}^{\lambda_2}\normBig{\frac{d}{d \lambda}\yl}d\lambda+\int_{\lambda=\lambda_1}^{\lambda_2}\frac{1}{\lambda}\normBig{ \dl }d\lambda.
	\end{align*}
	We also have $\norm{\dl} \leq \frac{L_f}{\lambda}$ from \Cref{lem:prox-movement-bound}. Substituting back and using $\lambda \ge \lambda_1$, we obtain
	\begin{equation*}
			\normBig{\dl[\lambda_2]-\dl[\lambda_1]} \le (R+ \frac{\Lf}{\lambda}) \int_{\lambda_1}^{\lambda_2} \frac{d \lambda}{\lambda} \le \prn*{ R + \frac{\Lf}{\lambda_1}} \log \frac{\lambda_2}{\lambda_1}
	\end{equation*}
	as claimed.
\end{proof}

\begin{proof}[Proof of \Cref{prop:bisection-gen}]
	We first prove the correctness of our procedure, recalling the notation $\hxl = \prox{f} (\yl)$,  $\hDl = \norm{\hxl - \yl}$ and $\Dl = \norm{\oracles[\lambda,\delta](\yl) - \yl}$, where $\delta = \frac{r}{17}$ throughout. 
	Our analysis will use the following fact: whenever $\Dl < r - \delta$ then (by \Cref{lem:prox_oracle_dist}) we have $\norm{\bprox{f}(\yl) -\yl} < r$ and consequently $\bprox{f}(\yl)=\prox{f}(\yl)=\hxl$ and $\abs{\hDl-\Dl} \le \delta$. 
	
By $\lambda_{\max} \ge \frac{2\Lf}{r}$ and \Cref{lem:prox-movement-bound}, we have $\hDl[\lambda_{\max}] \le r/2$ and $\bprox[\lambda_{\max},r]{f}(\yl)=\prox[\lambda_{\max}]{f}(\yl)$. By \Cref{lem:prox_oracle_dist} we have $\|\oracles[\lambda_{\max},\delta](\yl[\lambda_{\max}])-\hxl[\lambda_{\max}]\| \le \delta =  r/17$, and consequently $\Dl[\lambda_{\max}]\leq r/2+r/16\le 13r/16$. Therefore, the first while loop (\Cref{line:while1start}) executes at least once.
	
	Suppose the procedure terminates on Line~\ref{line:ls_lower_boundary}. Denoting the return value as $\lambda$, by the condition on \Cref{line:while1start} we have $\Dl \leq \frac{7r}{8}$, which by \Cref{lem:prox_oracle_dist} implies $\hDl < \frac{15r}{16} < r$. Consequently, outcome \ref{item:outcome-small-lambda} in \Cref{prop:bisection-gen} occurs. 

Next, consider the case where the procedure terminates due to the condition $\Dl \in [\frac{13r}{16}, \frac{15r}{16}]$ (either in \cref{line:stop-middle} or in \cref{line:while2start}). Applying \Cref{lem:prox_oracle_dist}, we conclude that $\hDl \in [\frac{3r}{4}, r)$. Consequently, outcome \ref{item:outcome-large-move} in \Cref{prop:bisection-gen} occurs. 

	It remains to check the case where the procedure terminates due to the condition $\log \frac{\lambda_u}{\lambda_{\ell}} < \frac{r}{8(R+\Lf/\lambda_{\ell})}$ in \cref{line:while2start}. Note that the binary search maintains the invariant $\Dl[\lambda_{\ell}] > \frac{15r}{16}$ and $\Dl[\lambda_u] < \frac{13r}{16}$. By \Cref{lem:prox_oracle_dist}, this implies 
	\[
	\hDl[\lambda_{\ell}] > \frac{7r}{8} > \hDl[\lambda_u].
	\]
	By continuity of $\hDl$, we therefore always have $\hDl[\lambda'] = \frac{7r}{8}$ for some $\lambda'\in (\lambda_\ell, \lambda_u)$. Therefore, if $\norm{x-v}\le 2R$, \Cref{lem:lambda_change} guarantees that
	\begin{equation*}
		\abs*{\hDl[\lambda_m] - \tfrac{7}{8} r} = \abs*{\hDl[\lambda_m] - \hDl[\lambda'] } \le \prn*{R+\frac{\Lf}{\min\{\lambda', \lambda_m\}}}\log\abs*{\frac{\lambda_m}{\lambda'}}
		\le 
		\half\prn*{R+\frac{\Lf}{\lambda_{\ell}}}\log{\frac{\lambda_u}{\lambda_{\ell}}}
	\end{equation*}
	Consequently, when $\log \frac{\lambda_u}{\lambda_{\ell}} < \frac{r}{8(R+\Lf/\lambda_{\ell})}$ and $\norm{x-v}\le 2R$, we are guaranteed that $\hDl[\lambda_m]\in (3r/4, r)$, and outcome \ref{item:outcome-large-move} in \Cref{prop:bisection-gen} occurs, concluding the proof of correctness.
	
	Next, we briefly justify the bounds on the $\lambda'$ values with which we query the \BOO in the $\linesearch$ procedure. 
	By construction, we have $\lambda' \in [\lambda_{\min},\lambda_{\max}]$. Let $\lambda_o$ and $\lambda_s$ be the procedure's output and smallest queried $\lambda$ values, respectively.  When terminating on lines \ref{line:ls_lower_boundary} or \ref{line:stop-middle}, we clearly have $\lambda_s = \lambda_o$. Furthermore, when terminating on \cref{line:while2start} $\lambda_s$ equals $\lambda_{\ell}$ when entering \cref{line:while2start} for the first time, and consequently $\lambda_o \le 2\lambda_s$. 
	
	Finally, we bound the total number of \BOO queries. The first while loop requires at most $\log_2 \frac{\lambda_{\max}}{\lambda_{\min}} $ queries. In the second while loop, initially we have $\log_2 \frac{\lambda_u}{\lambda_{\ell}} = 1$ and each query decreases $\log_2 \frac{\lambda_u}{\lambda_{\ell}}$ by a factor of 2. Therefore, using $\lambda_{\ell} \ge \lambda_{\min}$, the stopping condition $\log \frac{\lambda_u}{\lambda_{\ell}} < \frac{r}{8(R+\Lf/\lambda_{\ell})}$ must hold after $O\prn*{\log \prn*{\frac{R + \Lf/\lambda_{\min}}{r}}}$ queries. Given $\lambda_{\max} = \tfrac{2\Lf}{r}$, $\lambda_{\min} = \tfrac{\eps}{6rR}$ the total number of \BOO queries is bounded as
	\[\log\prn*{\frac{\lambda_{\max}}{\lambda_{\min}}}+\log \prn*{\frac{R + \Lf/\lambda_{\min}}{r}} = O\prn*{\log \frac{\Lf R^2}{r\epsilon}+\log \frac{\Lf R^2}{r\epsilon}} = O\prn*{\log \frac{\Lf R^2}{r\epsilon}}\] queries, giving the claimed complexity bound.

	\end{proof}

\section{\BOO implementation}\label{sec:boo-impl-proofs}

\subsection{Proof of \Cref{lem:bef-approx}}\label{ssec:bef-approx-proof}

Recall that the ``eponentiated softmax'' function is defined as follows
\begin{equation*}
	\beF(x) = {\epsilon'} \cdot \exp\prn*{\frac{\Fsm^{\lambda}(x)-\Fsm^{\lambda}(\bx)}{{\epsilon'}}} =  \sum_{i\in [N]} p_i(\bx) \gamma_i(x)~\mbox{where}~
	\gamma_i(x)\defeq \epsilon' e^{ \frac{f_i^\lambda(x)-f_i^{\lambda}(\bx)}{{\epsilon'}}},
\end{equation*}
and note that $\beF(\bx) = \epsilon'$.

\beFapprox*

\begin{proof}
For the first statement, we note that $\beF$ is a monotonic increasing transformation of $\Fsm^{\lambda}$ and consequently they have the same minimizer $\xopt$ in $\ball_{r}(\bx)$. Further
\begin{flalign*}
 \Fsm^{\lambda}(x)-\Fsm^{\lambda}(\xopt) & = \eps'\log\prn*{\frac{\beF(x)}{\beF(\xopt)}} = \beF(\bx)\log\prn*{1+\frac{\beF(x)-\beF(x\opt)}{\beF(\xopt)}}
\\  & 
\le \frac{\beF(\bx)}{\beF(\xopt)} \prn*{ \beF(x) - \beF(\xopt)},
\end{flalign*}
where the final inequality uses $\log(1+x) \le x$. Next, note that Lipschitz continuity of each $f_i$ implies
\begin{equation*}
	f_i^\lambda(\xopt) \ge f_i(\xopt) \ge f_i(\bx) - \Lf \norm{\xopt-\bx} = f_i^\lambda(\bx) - \Lf \norm{\xopt-\bx}.
\end{equation*}
Substituting $\norm{\xopt-\bx} \le r \le c\epsilon' / \Lf$ and $C\ge e^{c}$, we have that 
\begin{equation*}
	e^{f_i^\lambda(\xopt)/\epsilon'} \ge  e^{f_i^\lambda(\bx)/\epsilon' - c} \ge \frac{1}{C} e^{f_i^\lambda(\bx)}.
\end{equation*}
Consequently, we have
\begin{equation*}
	\frac{\beF(\bx)}{\beF(\xopt)} = \frac{\sum_{i\in [N]} e^{f_i^\lambda(\bx)/\epsilon'}}{\sum_{i\in [N]} e^{f_i^\lambda(\xopt)/\epsilon'}} \le C,
\end{equation*}
proving the first statement.

For the second statement, we first compute the gradient and Hessian of the function $\gamma_i(x)$ as
\begin{align*}
	\grad \gamma_i (x) & = \exp\left(\frac{\lambda}{2\eps'}\|x-\bx\|^2\right)
	\exp\prn*{\frac{f_i(x)-f_i(\bx)}{\eps'}}\left[\grad f_i(x)+\lambda(x-\bx)\right],~\text{and}\\
	\grad^2 \gamma_i(x) & = \exp\prn*{\frac{f_i(x)-f_i(\bx)}{\eps'}}H_i,~\text{where}\\
 H_i & = \grad^2 f_i(x)+\lambda I+\frac{1}{\eps'}\grad f_i(x)\grad f_i(x)^\top + \frac{\lambda^2}{\eps'}(x-\bx)(x-\bx)^\top.
\end{align*}
(Note that the expression for $\hess \gamma_i$ assumes that $f_i$ is twice differentiable almost everywhere; we only rely on it for \Cref{ssec:asvrg} where this holds.) 

Now we note that for any $i\in[N]$
\begin{equation*}
\|\grad \gamma_i(x)\|  \le \exp(\lambda r^2/2\eps'+L_fr/\eps')(\lambda r+\Lf),
\end{equation*}
and
\begin{equation*}
\lambda e^{-\lambda r^2/2\eps'-L_fr/\eps'} I \preceq \grad^2\gamma_i(x) \preceq e^{\lambda r^2/2\eps'+L_fr/\eps'}\left(\Lg+\lambda +\Lf^2/\eps'+\lambda^2r^2/\eps'\right)I,
\end{equation*}
where for the inequality we use the fact that $f_i$ is $\lip$-Lipschitz, $\Lg$-smooth, and that $ x\in \ball_{r}(\bx)$.

Now by plugging in the assumption that $r\le c\eps'/\Lf$ and $\lambda\le c\Lf/r$ we have $\lambda r\le c\Lf$, $\lambda^2r^2/\eps'\le c^2\lambda$, and  $\exp(\lambda r^2/2\eps'+L_fr/\eps')\le \exp(c+c^2/2)$. Thus we can further simplify the bounds on the gradient and Hessian of $\gamma_i(x)$ by definition of $C$ as  
\[
\|\grad \gamma_i(x)\|\le C\Lf
\quad 
\text{and}
\quad
C^{-1}\lambda I\preceq \grad^2\gamma_i(x)\preceq C\left(\Lg+\lambda +\Lf^2/\eps'\right) I,
\]
which completes the proof.
\end{proof}

\subsection{SGD implementation}\label{ssec:sgd-app}

We first cite the following stochastic gradient method with restarts that obtains an $\widetilde{O}(1/\mu T)$ bound for $\mu$-strongly convex function.

\begin{lem}[Theorem 11 in \citet{hazan2014beyond}]\label{lem:SMD} Given a $\mu$-strongly-convex objective function $f:\mathcal{X}\to \R$ with minimizer $\xopt$ with an unbiased stochastic estimator with norm at most $G$ in the convex compact set $\mathcal{X}$, Epoch-SGD-Proj algrotihm finds an approximate minimizer $\tilde{x}$ satisfying with probability $1-\sigma$

\[
f(\tilde{x})-f(\xopt)\le O\left(\frac{G^2\log(\log (T)/\sigma)
}{\mu T}\right),
\]
using $T$ stochastic gradient queries.
\end{lem}

\begin{algorithm2e}
	\caption{Epoch-SGD-Proj on the exponentiated softmax}
	\label{alg:innerloop-SGD}
	\LinesNumbered
	\DontPrintSemicolon
	\KwInput{Functions $f_1,\ldots,f_N$, ball center $\bx$, ball radius $\reps$, regularization strength $\lambda$, smoothing parameter $\eps'$, failure probability $\sigma$}
	\KwParameter{Step size $\eta_1=1/(3\lambda)$, domain size $D_1=\Theta(G\sqrt{\log(\log(T)/\delta)}/\lambda)$, $T_1=450$ and total iteration budget $T$}
	\KwOutput{Approximate minimizer of $\beF$ (and hence $\Fsm^\lambda$) in $\ball_{\reps}(\bx))$}
	Precompute sampling probabilities $p_i = e^{f_i(\bx)/\eps'}/\sum_{i\in[N]}e^{f_i(\bx)/\eps'}$ for all $i\in[N]$\;
	Initialize $x_1^1\in \ball_{\reps}(\bx)$ arbitrarily, set $k=1$ \;
	\While{$\sum_{i\in[k]}T_i\le T$}
	{
	\For{$t = 1, \ldots, T_k$}
		{
			Sample $i\in[N]$ with probability $p_{i}$\;
			Query stochastic gradient $\hat{g}_t = e^{(f_i^\lambda(x)-f_i^\lambda(\bx))/\eps'}\grad f_i^\lambda(x)$\;
			Update $x^k_{t+1} \leftarrow \Pi_{\ball_{\reps}(\bx)\cap\ball_{D_k}(x_1^k)}(x_t^k-\eta_k \hat{g}_t)$\;
		}
	Let $x_1^{k+1} \leftarrow  \frac{1}{T_k}\sum_{t\in[T_k]}x_t^k$\;
	Update parameters $T_{k+1} \leftarrow  2T_k$, $\eta_{k+1} \leftarrow  \eta_k/2$, $D_{k+1}\leftarrow D_k/\sqrt{2}$, $k \leftarrow  k+1$
	}
	\Return $x_1^k$
\end{algorithm2e}

Applying the lemma immediately gives the following guarantee for minimizing $\beF$ inside a ball of radius $\reps$ and hence implementing an $\reps$-BROO for $\Fsm$.

\corsgd*

\begin{proof}
We first note that by choice of $\reps$ and bounds on $\lambda$, $\beF$ is $\Omega(\lambda)$-strongly convex according to Lemma~\ref{lem:bef-approx}.
	At each iteration, we sample $i\in[N]$ with probability $p_i$ and compute stochastic gradient 
	\[\grad\gamma_i(x) = \gamma_i(x)\left(\frac{\lambda}{\eps'}(x-\bx)+\grad f_i(x)/\eps'\right)~~\text{bounded by}~~G = O(\Lf)\] following from the second statement of \Cref{lem:bef-approx}. Thus by directly applying Lemma~\ref{lem:SMD} with $T = \Theta(\Lf^2\lambda^{-2}\delta^{-2}\log(\log(\Lf/\lambda\delta)/\sigma))$ the algorithm outputs an approximate minimizer $\tilde{x}$ satisfying 
	\[
 \beF(\tilde{x})-\min_{x\in\ball_{\reps}(\bx)}\beF(x)\le O\left(\frac{\Lf^2}{\lambda T}\log(\log (T)/ \sigma) \right)\le \frac{\lambda\delta^2}{6e^2}.
\]

By the first property of \Cref{lem:bef-approx}, bound above implies that for $\xopt = \bprox{\Fsm}$
\[
\Fsm^\lambda(x)-\Fsm^\lambda(\xopt)\le 3e^2(\beF(x)-\beF(\xopt))\le \frac{\lambda\delta^2}{2},
\]
i.e., algorithm outputs a valid $\reps$-BROO response for $\Fsm$.

The bound on the total computational cost follows from noticing that the initialization cost is dominated by $N$ function value queries, and that the cost of each step in the stochastic gradient descent is dominated by a function and a gradient query at the current iteration.

\end{proof}

\subsection{Accelerated variance reduction implementation}\label{ssec:asvrg-app}

We first cite the following accelerated variance reduction guarantee.

\begin{lem}[Theorem 5.4 in \citet{allen2016katyusha}]\label{lem:ASVRG}
Let $f_1,\ldots,f_N$ be $L$-smooth and $\mu$-strongly-convex, let $F(x) = \sum_{i\in[n]} w_i f_i(x)$ with $w_i\ge0$ and $\sum_{i\in[N]}w_i=1$, and let $\xopt \in \argmin F(x)$. For any $s\in \N$, Katyusha1 with batch size $b=1$ and initial point $\bar{x}$ finds an approximate solution $\tilde{x}_s$  satisfying
\[
\E \brk*{F(\tilde{x}_s)-F(\xopt)}\le \frac{1}{2^s}\left[F(\bar{x})-F(\xopt)\right]
\]
using 
\[
O\left(s\cdot \left(N+\sqrt{\frac{N\cdot L}{\mu}}\right)\right)
~\mbox{evaluations of }\grad f_i(x).
\]
\end{lem}

An immediate corollary of \Cref{lem:ASVRG} provides a high-probability guarantee.

\begin{corollary}\label{cor:asvrg-helper}
	Under the same assumptions as \Cref{lem:ASVRG}, for any $\eps>0$ and $\sigma\in(0,1)$ Katyusha1 with batch size $b=1$ and initial point $\bar{x}$ finds an approximate solution $\tilde{x}$ satisfying
\[
F(\tilde{x})-F(\xopt)\le \eps~\mbox{with probability at least }1-\sigma
\]
using
\[
O\left(\left(N+\sqrt{\frac{N\cdot L}{\mu}}\right)\log\left(\frac{F(\bar{x})-F(\xopt)}{\veps\sigma}\right)\right)
~\mbox{evaluations of }\grad f_i(x).
\]
\end{corollary}

\begin{proof}
	Since $F(\tilde{x}_s)-F(\xopt)\ge0$, Markov's inequality and \Cref{lem:ASVRG} imply that for any $\eps>0$, 
	\[
	\Pr*(F(\tilde{x}_s)-F(\xopt)\ge \eps)\le \frac{\E\left[F(\tilde{x}_s)-F(\xopt)\right]}{\eps}\le \frac{(1/2)^s\left(f(\bar{x})-f(\xopt)\right)}{\eps},
	\]
	and thus the high-probability bound follows .
\end{proof}

Specializing \Cref{cor:asvrg-helper} for $\beF(x) = \sum_{i\in[N]} p_i(\bx) \gamma_i(x)$, we obtain the following guarantee.

\corasvrg*

\begin{proof}

By the choice of $\reps$ and the bound $\lambda=O(\Lf^2 /\epsilon')$, \Cref{lem:bef-approx} guarantees that the $\gamma_i$ have strong convexity $\mu = \Omega(\lambda)$, and smoothness $L =   O(\Lg + \lambda + \Lf^2/\eps' +\lambda^2\reps^2/\eps') = O(\Lg+ \Lf^2/\eps')$. In addition, for  $\xopt\defeq \arg\min_{x\in\ball_{\reps}(\bx)}\beF(x)$ one has
\begin{align*}
\beF(\bar{x})-\beF(\xopt)& \le \langle\grad_x \beF(\xopt),\bar{x}-\xopt\rangle\\
&   \le \|\grad_x \beF(\xopt)\|\|\bar{x}-\xopt\|= O(\Lf \reps),
\end{align*}
where we use the second property for bounding $\grad_x \beF(\bar{x})$ from \Cref{lem:bef-approx}. Plugging these into the complexity of \Cref{lem:ASVRG} and noticing that each evaluation of $\grad \gamma_i$ requires evaluation of $f_i$ and $\grad f_i$ (assuming $f_i(\bx)$ is pre-stored) gives the stated complexity bound.
\end{proof}

\subsection{Proof of \Cref{thm:ub}}\label{ssec:ub-proof}
With \Cref{cor:sgd,cor:asvrg} established, we prove \Cref{thm:ub}.

\thmub*

\begin{proof}
We use guarantees of \Cref{thm:ms-bacon-redux} and \Cref{cor:sgd,cor:asvrg} on the problem $\min_{x}\Fsm(x), \|x-x_0\|\le O(R)$ to prove the correctness and bound the complexity of the algorithm.

\paragraph{Correctness.}

We first note that \Cref{thm:ms-bacon-redux} guarantees that we only make BROO calls with $\lambda \le O(\Lf / \reps)$, making \Cref{cor:sgd,cor:asvrg} applicable.
Let 
\begin{equation}\label{eq:oracle-call-number-bound}
	T = O\prn*{
		\prn*{ \frac{R}{\reps}}^{2/3}
		 \log^2\prn*{\frac{\Lf R^2}{\reps\epsilon}}
	}
\end{equation}
be the upper bound on the total number of oracle calls guaranteed in \Cref{thm:ms-bacon-redux}. Taking $\sigma = \frac{1}{100T}$ in \Cref{cor:sgd,cor:asvrg} and applying a union bound, we see that with probability at least ${99}/{100}$, the outputs of the corresponding BROO implementations are valid throughout the execution of \Cref{alg:ms-bacon-redux}. Consequently by \Cref{thm:ms-bacon-redux} we have that \Cref{alg:ms-bacon-redux} with accuracy $\eps/2$ outputs $x_o$ such that $\Fsm(x_o)-\min_{x:\|x-x_0\|\le R}\Fsm(x) \le \epsilon/2$. Using the fact that $0 \le \Fsm(x) - \F(x) \le \epsilon/2$ for all $x\in\R^d$ \cite[see, e.g.,][Lemma 45]{carmon2020acceleration}, we conclude that $\F(x_o)-\min_x\F(x) = \F(x_o)-\min_{\{x:\|x-x_0\|\le R\}}\F(x) \le \Fsm(x_o)-\min_{x:\|x-x_0\|\le R}\Fsm(x)+\eps/2\le \eps $, establishing correctness.

\paragraph{Complexity.} 
Substituting $\reps=\epsilon/(2\Lf\log N)$ into~\eqref{eq:oracle-call-number-bound}, we see that the total number of oracle calls is
\begin{equation*}
	T = O\prn*{
		\prn*{ \frac{\Lf R \log N}{\reps}}^{2/3}
		 \log^2\prn*{\frac{\Lf R \log N}{\epsilon}}
	}.
\end{equation*}
To bound to complexity of the SGD implementation, we simply multiply $T$ by the per-call complexity bound~\eqref{eq:sgd-complexity-bound} where we substitute $\delta = \Omega(\epsilon / (\lambda R))$ as guaranteed by \Cref{thm:ms-bacon-redux} and $\sigma = 1/(100T)$. 

To bound the complexity of the accelerated variance reduction implementation, we similarly substitute $\delta = \Omega(\epsilon / (\lambda R))$, $\lambda = O(\Lf/\reps)$ and $\sigma = 1/(100T)$ into~\eqref{eq:asvrg-complexity-bound}. Summing the result over all oracle calls yields the complexity bound
\begin{equation*}
	O\prn*{
		\prn*{\Tf+\Tg}\log\prn*{\frac{\Lf R \log N}{\epsilon}}
		\prn*{NT + \sqrt{N}\prn*{\sqrt{\Lg} + \Lf\sqrt{\frac{\log N}{\epsilon}}}\sum_{i\in[T]} \frac{1}{\sqrt{\lambda\fancyind{i}}}
		}
	},
\end{equation*}
where $\crl{\lambda\fancyind{i}}$ is the sequence of $\lambda$ values with which \Cref{alg:ms-bacon-redux} calls the BROO implementation. 
\Cref{thm:ms-bacon-redux} guarantees that $\sum_{i \in [T]} \frac{1}{\sqrt{\lambda\fancyind{i}}} \leq O\prn[\big]{ \frac{R}{\sqrt{\eps}}\log\frac{\Lf R^2}{r\eps}}=O\prn[\big]{\frac{R}{\sqrt{\epsilon}}\log\frac{\Lf R\log N}{\epsilon}}$, completing the proof.
\end{proof}

\section{Lower bound proofs}\label{sec:app-lb}

\subsection{Proof of \Cref{prop:prog-control}}\label{sec:prog-control-proof}

In this section, we make frequent use of the indication notation $\indic{\cdot}$, where $\indic{A}=1$ if event $A$ holds and $\indic{A}=0$ otherwise.

\propProgControl*

\begin{proof}
	Let us define several quantities that are important for our proof. First, we track the maximum progress attained by the algorithm queries.
	\begin{equation*}
		p_t \defeq \prog(U^\top x_t)
		~~\mbox{and}~~
		\bp \defeq \max_{s \le t} p_s.
	\end{equation*}
	Next, we recursively define a sequence that tracks the algorithm's progress in ``unlocking'' the relevant elements of the finite sum,
	\begin{equation*}
		\B \defeq \indic{\iperm(i_t) = \C+ 1}
		~~\mbox{where}~~
		\C \defeq \min\crl*{\sum_{s < t} \B[s], T}.
	\end{equation*}
	To understand these definitions, note that $\C$ is the largest number $k$ such that $1,2,\ldots,k$ is a subsequence of $\iperm(i_1),\iperm(i_2),\ldots,\iperm(i_t)$. Therefore, a ``zero-respecting'' algorithm (satisfying $p_t \le \max_{s<t}\crl*{\prog(\grad f_{\iperm(i)}(U^\top x_s))}$) can only query points with progress at most $C_t$. We define the event that the general algorithm under consideration behaves as though it was zero-respecting for the first $t$ iterations as
	\begin{equation*}
		\zrEv \defeq \crl*{\bp[s] \le \C[s]~~\mbox{for all}~~s\le t}. 
	\end{equation*}

	We also define the stopping time
	\begin{equation*}
		\Theta_k \defeq \min\crl*{t \mid \C = k}
	\end{equation*}
	and the difference sequence
	\begin{equation*}
		\Delta_k \defeq \Theta_k - \Theta_{k-1}.
	\end{equation*}
	
	Let 
	\begin{equation*}
		\lbt \defeq \frac{1}{16}N \prn*{ T-\log\frac{2}{\delta} }
	\end{equation*}
	so that our goal is to prove that $\P\prn*{ \bp[\floor{\lbt}] < T } >1- \delta$. 
	Note that the intersection of the events $\C[\floor{\tau}] < T$ and  $\zrEv[NT] \subset \zrEv[\floor{\lbt}]$ imply the desired event $\bp[\floor{\lbt}] < T$. Moreover, $\C[\floor{\tau}] < T$ is equivalent to $\Theta_T > \tau$. Consequently, we can upper bound $\P\prn*{ \bp[\floor{\lbt}] \ge T }$ by the probability that $\zrEv[NT]$ does not occur plus the probability that both $\Theta_T \le \lbt$ and $\zrEv[NT]$ occur, i.e.,
	\begin{equation*}
		\P\prn*{ \bp[\floor{\lbt}] \ge T } \le 
		\Pr*( \zrEv[NT]^c  ) + \Pr*( \Theta_T \le \lbt , \zrEv[NT]).
	\end{equation*}
	Our strategy is to bound each of  $\Pr*( \zrEv[NT]^c  )$ and $\Pr*( \Theta_T \le \lbt , \zrEv[NT])$ by $\delta/2$.
	
	As one final piece of of notation, we write $U_{\le k}$ for the first $k$ columns of $U$, and $\perm\restrictedTo{k}$ as shorthand for $\perm(1),\ldots,\perm(k)$.
	\begin{lem}\label{lem:prog-control-key}
		For any $t \ge 1$ and $k\le T$, if the events $\zrEv[t-1]$ and $\C \le k$ hold, the oracle responses to queries $(i_s, x_s)_{s < t}$, as well as the queries $(i_s, x_s)_{s \le t}$,  are deterministic (measurable) functions of $\algseed$, $\perm\restrictedTo{k}$ and $U_{\le k}$. 
		Moreover, 
		\begin{enumerate}[label=(\alph*)]
			\item \label{item:prog-control-key-a} The random variable $\Theta_k \indic{\zrEv[\Theta_k]}$ is measurable w.r.t\ $\algseed, \perm\restrictedTo{k}$ and $U$.
			\item \label{item:prog-control-key-b} When $\zrEv[t-1]$ holds, $\indic{\C = k}$ is measurable w.r.t.\ $\algseed, \perm$ and $U_{\le k}$.
		\end{enumerate}
	\end{lem}
	
	\begin{proof}
		Consider the query $i_s, x_s$ for $s<t$; we first show that under $\zrEv[t-1]$ and $\C \le k$ we can compute the oracle response to this query using only $\perm\restrictedTo{k}$ and $U_{\le k}$. Since $ \zrEv[t-1]$ holds, we have that $\prog(U^\top x_s) = p_s \le \C[s]$.
		Invoking \Cref{def:zc} of the $N$-element zero chain, there exists a neighborhood of $x_s$ such that for all $y$ in that neighborhood we have 
		\begin{equation}\label{eq:zc-app}
			\tf_{i_s}(y) = f_{\iperm(i_s)}(U^\top y) = 
			\begin{cases}
				f_{\iperm(i_s)}(U_{\le \C[s]}^\top y)
				& \iperm(i_s) < \C[s] + 1 \\
				f_{\iperm(i_s)}(U_{\le \C[s]+1}^\top y)
				& \iperm(i_s) = \C[s] + 1 \\
				f_{N}(U_{\le \C[s]}^\top y)
				& \iperm(i_s) > \C[s] + 1. \\
			\end{cases}
		\end{equation}
		Recall that $\B[s] = \indic{\iperm(i_s) = \C[s]+1}$ and that $\C[s+1]=\min\{\C[s]+\B[s],T\}\le \C[t]\le k$ and let
		\begin{equation*}
			\iperm\restrictedTo{k}(j) = \begin{cases}
				\iperm(j) & j\in \perm([k]) \\
				N & \mbox{otherwise}.
			\end{cases}
		\end{equation*}
		The relationship~\eqref{eq:zc-app} implies that
		\begin{equation*}
			\tf_{i_s}(y)  = f_{\iperm\restrictedTo{k}(i_s)}(U_{\le k}^\top y)
		\end{equation*} 
		for all $y$ in the neighborhood of $x_s$.
		
		The above discussion shows that (under $\zrEv[t-1]$) the oracle response to $(i_s,x_s)$ is measurable with respect to $i_s, x_s, \perm\restrictedTo{k}$ and $U_{\le k}$, for every $s<t$. Moreover, $(i_1, x_1)$ is  measurable with respect to $\algseed$, and consequently the first oracle response and the second query $(i_2, x_2)$ are measurable with respect to $\algseed, \perm\restrictedTo{k}$ and $U_{\le k}$. Repeating this argument inductively shows that (under $\zrEv[t-1]$ and $C_t \le k$) we can generate the entire query sequence $(i_s, x_s)_{s\le t}$, as well as the oracle responses to all but the last query, from $\algseed, \perm\restrictedTo{k}$ and $U_{\le k}$, as claimed.
		
		To show part \ref{item:prog-control-key-a} of the lemma, note that with access to the entire matrix $U$, after generating queries $x_1, \ldots, x_s$ we can test whether $\zrEv[s]$ holds. Moreover, the knowledge of $\perm\restrictedTo{k}$ suffices to test (for any query sequence $i_1, i_2, \ldots$) whether $C_t \ge k$ for every $t$. Therefore, using $\algseed, \perm\restrictedTo{k}$ and $U_{\le k}$ we can iteratively compute $(i_s, x_s)_{s\le t}$ until arriving at an iterate $t$ where either (a) $\C < k$ but $\zrEv[t]$ does not hold (which implies that $\zrEv[\Theta_k]$ fails as well), or (b) $\C = k$ (in which case we have found $\Theta_k$). In either case, we know the value of $\Theta_k \indic{\zrEv[\Theta_k]}$.
		
		Finally, to show part \ref{item:prog-control-key-b} of the lemma, note that with full knowledge of $\perm$ we can track the values of $\C[1],\ldots, \C$ for any sequence of queries $i_1, \ldots, i_{t-1}$. Therefore, given $\zrEv[t-1]$, we may use $\algseed, \perm\restrictedTo{k}$ and $U_{\le k}$ to iteratively generate queries until either (a) we arrive at an iteration $s < t$ where $\C[s] > k$, in which case $\indic{\C=k}=0$, or (b) we successfully generate  iteration $t-1$ in which case we can compute $\C$. In either case, we know the value of $\indic{\C=k}$.
	\end{proof}
	
	\begin{lem}\label{lem:delta-k-lb}
		For every $k\le T$, and $j \ge 0$, we have
		\begin{equation*}
			\Pr*( \Delta_k \le j, \zrEv[\Theta_{k}] | \algseed, \perm\restrictedTo{k-1}, U ) \le \frac{j}{N-k+1}.
		\end{equation*}
	\end{lem}
	\begin{proof}
		Suppose the algorithm could access the random permutation $\perm$ via an alternative oracle that, when queried at index $i\in [N]$ returns the number $\iperm(i)$. We say that the algorithm succeeds if one of the first $j$ queries is $\perm(k)$ (so that the oracle returns $k$). 
		Given $\perm\restrictedTo{k-1}$, the random variable $\perm(k)$ is uniformly distributed over $\mathcal{I}=[N]\setminus \perm([k-1])$. Therefore, for any set $\mathcal{J}$ of $j$ queries, we have
		\begin{equation*}
			\Pr*( \perm(k) \in \mathcal{J} | \perm\restrictedTo{k-1} )
			= \frac{\E \abs*{\mathcal{J} \cap \mathcal{I}}}{\abs*{\mathcal{I}}} \le
			\frac{\E \abs*{\mathcal{J}}}{\abs*{\mathcal{I}}} = \frac{j}{N-k+1}
		\end{equation*}
		and so for every algorithm the probability of success is at most $j/(N-k+1)$.

		Now return to the original problem and the original oracle, and note that for $\Delta_k \le j$ to hold, we must have $\iperm(i_t) = k$ for some $t$ in $\{\Theta_{k-1}, \ldots, \Theta_{k-1}+j-1\}$, corresponding to ``success'' in the problem described above. Moreover, when the event $\zrEv[\Theta_{k}]$ holds, \Cref{lem:prog-control-key} guarantees that the oracle responses to the queries at iteration $1, \dots, \Theta_{k-1}-1$ are deterministic functions of $\algseed, U$ and $\perm\restrictedTo{k-1}$ (since $C_{\Theta_{k-1}} = k-1$ by definition). 
		Consequently, when $\zrEv[\Theta_{k-1}]$ holds, the true optimization algorithm operates with no more information that the alternative oracle setting described above, giving the claimed probability bound.
	\end{proof}

	\begin{lem}\label{lem:theta-high-prob}
		We have
		\begin{equation*}
			\Pr*( \Theta_T \le \tfrac{1}{16} N(T-\log\tfrac{2}{\delta}) , \zrEv[NT] ) 	\le \frac{\delta}{2}.
		\end{equation*}
	\end{lem}
	\begin{proof}
		Let $T_{\pm\delta} \defeq \half \prn*{T \pm\log\frac{2}{\delta}} $. We begin with a sequence of straightforward inequalities:
		\begin{flalign*}
			\Pr( \Theta_T \le \tfrac{1}{8} N T_{-\delta} , \zrEv[NT] ) \hspace{-96pt} &\hspace{96pt}
			\le 
			\Pr( \Theta_T \le \tfrac{1}{8}  N T_{-\delta} , \zrEv[\Theta_T] )
			\\ &\le 
			\Pr*( \sum_{k \in [T]} \indic{\Delta_k 
				>  \tfrac{1}{8} N} \le  T_{-\delta} , \zrEv[\Theta_T] )
			\\&	=
			\Pr*( \sum_{k \in [T]} \indic{\Delta_k 
				\le   \tfrac{1}{8} N} > T_{+\delta} , \zrEv[\Theta_T] )
			\\&\le
			\Pr*( \sum_{k \in [T]} \indic{\Delta_k 
				\le   \tfrac{1}{8} N, \zrEv[\Theta_{k}]}>  T_{+\delta} )
			\\&\le 
			e^{-2T_{+\delta}} \Ex*{ e^{ 2\sum_{k \in [T]} \indic{\Delta_k 
						\le   \frac{1}{8} N, \zrEv[\Theta_{k}]}} }, \numberthis
			\label{eq:prog-control-chernoff-bound}
		\end{flalign*}
		where the last transition is an application of the Chernoff bound (with parameter $\lambda=2$). 
		
		By \Cref{lem:prog-control-key}, the random variable $\indic{\Delta_k 
			\le   \frac{1}{8} N, \zrEv[\Theta_{k}]}$ is measurable with respect to $\algseed, \perm\restrictedTo{k}$ and $U$ since it is a function of $\Theta_{k} \indic{ \zrEv[\Theta_{k}] }$ and $\Theta_{k-1} \indic{ \zrEv[\Theta_{k-1}] }$. Therefore,
		\begin{flalign*}
			\hspace{2em} & \hspace{-2em} \Ex*{ e^{ 2\sum_{k \in [T]} \indic{\Delta_k 
						\le   \frac{1}{8} N, \zrEv[\Theta_{k}]}} }
			\\ &= \Ex*{
				\prn*{ 1 + (e^2 - 1)\Pr*( \Delta_T \le \tfrac{1}{8}, \zrEv[\Theta_{T}] | \algseed, \perm\restrictedTo{T-1}, U ) }e^{ 2\sum_{k \in [T-1]} \indic{\Delta_k 
						\le   \frac{1}{8} N, \zrEv[\Theta_{k}]}} }
			\\ & \le \prn*{1+\frac{1}{4}(e^2 -1)}  
			\Ex*{e^{ 2\sum_{k \in [T-1]} \indic{\Delta_k 
						\le   \frac{1}{8} N, \zrEv[\Theta_{k}]}}},
		\end{flalign*}
		where the inequality follows from~\Cref{lem:delta-k-lb} with $j=N/4$ and $k=T\le N/2$. Noting that $1+\frac{1}{4}(e^2 -1) \le e$ and iterating this argument, we conclude that
		\begin{equation*}
			\Ex*{ e^{ 2\sum_{k \in [T]} \indic{\Delta_k 
						\le   \frac{1}{8} N, \zrEv[\Theta_{k}]}} } \le e^T.
		\end{equation*}
		Substituting this bound into \cref{eq:prog-control-chernoff-bound} and recalling that $e^{-2T_{+\delta}} = \frac{\delta}{2} e^{-T}$ concludes the proof.
	\end{proof}
	
	\begin{lem}\label{lem:zr-high-prob}
		We have
		\begin{equation*}
			\Pr*( \zrEv[NT]^c ) \le \frac{\delta}{2}. %
		\end{equation*}
	\end{lem}
	
	\begin{proof}
		By definition, we have
		\begin{equation}\label{eq:zr-fail-prob-breakdown}
			\Pr*( \zrEv[NT]^c ) = \sum_{t \in [NT]} \Pr*( p_t > \C , 
			\zrEv[t-1] ).
		\end{equation}
		We fix $t \le NT$ and argue that $\Pr*( p_t > \C, \zrEv[t-1]) \le \frac{\delta}{2NT}$. 
		Let $k\le T$; by \Cref{lem:prog-control-key} events $\zrEv[t-1]$ and $\C = k$ hold, we have $x_t = h(\algseed, \Pi, U_{\le k})$ for some measurable functions $h$. Therefore, we have
		\begin{flalign*}
			\hspace{2em} & \hspace{-2em}
			\Pr*( p_t > \C , \zrEv[t-1]  | \C=k, \algseed, \perm, U_{\le k} )
			\\ & \overeq{(i)}
			\Pr*( p_t > k , \zrEv[t-1]  | \algseed, \perm, U_{\le k} )
			\\ & =
			\Pr*( \prog(U^\top x_t) > k , \zrEv[t-1]  | \algseed, \perm, U_{\le k} )
			\\ & =
			\Pr*( \prog(U^\top h(\algseed, \Pi, U_{\le k})) > k , \zrEv[t-1]  | \algseed, \perm, U_{\le k} )
			\\ & \le 
			\Pr*( \prog(U^\top h(\algseed, \Pi, U_{\le k})) > k  | \algseed, \perm, U_{\le k} )
			\\ & \overle{(ii)}
			\sum_{j=k+1}^{T} \Pr*( \abs[\big]{u_j^\top h(\algseed, \Pi, U_{\le k})} >\alpha  | \algseed, \perm, U_{\le k} ),
		\end{flalign*}
		where $(i)$ follows from part \ref{item:prog-control-key-b} of \Cref{lem:prog-control-key}, and $(ii)$ follows from the definition of $\prog$ and a union bound, where we write $u_j$ for the $j$th column of $U$. 

		Conditionally on $\algseed, \perm, U_{\le k}$, $u_j$ is uniform over the unit ball in the subspace of $\R^{d}$ orthogonal to $U_{\le k}$. Moreover, the magnitude of the projection of $h(\algseed, \perm, U_{\le k})$ to that subspace can be at most $1$, since $x_t$ is a unit vector. Therefore, the probability that $\abs[\big]{u_j^\top h(\algseed, \perm, U_{\le k})} > \alpha$ holds is at most the probability that a coordinate of a uniform unit vector in $\R^{d-k}$ has magnitude greater than $\alpha$. By standard concentration of measure arguments, this probability is at most $2\exp(-\frac{d-k+1}{2\alpha^2})$. By our choice of $d$ (and recalling $k\le T$), we have that
		\begin{equation*}
			\Pr*( p_t > \C , \zrEv[t-1]  | \C=k, \algseed, \perm, U_{\le k} ) \le 
			(T-k) \cdot \frac{\delta}{2NT^2} \le \frac{\delta}{2NT},
		\end{equation*}
		and substituting back into~\cref{eq:zr-fail-prob-breakdown} concludes the proof.
	\end{proof}

\end{proof} %
\subsection{Proof of \Cref{lem:hard-instance-props}}\label{sec:hard-instance-props-proof}

\lemHardInstanceProps*

\begin{proof}
	To see why part \ref{item:hi-zc} holds, fix $x\in \R^T$ and let $p = \prog[\alpha_T](x)$. 
	First, we have 
	 $\fhard_i(x) = \fhard_i(x\coind{\le i})$  for every $i$ and $x$, which immediately gives the first two cases of ~\cref{eq:zc-def}. Second, when $i > p+1$, we have $|x\coind{i} - x\coind{i-1}| < 2\alpha_T$ and therefore $\fhard_i(x) = \linkfun(t)\indic{i\le N}$ for some $|t| < \alpha$. Consequently $\fhard_i$ is identically zero in a neighborhood of $x$, giving the third and final case in~\cref{eq:zc-def}.
	
	Part \ref{item:hi-lip} is immediate because $\fhard_i$ is a composition of a 1-Lipschitz and $\nsm$-smooth function with the linear transformation $(x\coind{i}-x\coind{i-1})/2$ which has operator norm smaller than 1.
	
	Finally, to see part \ref{item:hi-optgap}, first note that the global minimum of $\Fhard$ satisfies $\Fhard(\xopt) = 0$ and $(\xopt)\coind{i} = \frac{1}{\sqrt{T}}$ for every $i\le T$ (and therefore has unit norm). Consider any $x$ for which $\prog[\alpha_T](x) < T$, so that $x\coind{T} \le \alpha_T \le \frac{1}{\sqrt{T}}$. We have
	\begin{equation*}
		\max_{i\le T} \abs*{ x\coind{i-1} - x\coind{i} } \ge \frac{1}{T} \sum_{i \in [T]} 
		\abs*{ x\coind{i-1} - x\coind{i} } \ge \frac{1}{T} \abs*{x\coind{0} - x\coind{T}} \ge \frac{1}{T} \prn*{\frac{1}{\sqrt{T}} - \alpha_T} \ge \frac{3}{4T^{3/2}}.
	\end{equation*} 
	Since $\linkfun(t)$ is non-decreasing in $|t|$, we have $\Fhard(x) = \linkfun\prn*{\half \max_{i\le T} \abs*{ x\coind{i-1} - x\coind{i}}}$, and consequently $\Fhard(x) \ge \linkfun\prn*{\frac{3}{8T^{3/2}}}$. To obtain the final bound, we observe that $\linkfun[\alpha,\nsm](t) \ge \min\crl*{ \half (t- \alpha), \frac{\ell}{2} (t-\alpha)^2 }$ for any $t\ge \alpha$.
\end{proof} %
\subsection{Proof of \Cref{thm:lb}}\label{sec:lb-proof}

\thmLB*

\begin{proof}
	We first show that an $\Omega \prn[\big]{  
		N \brk[\big]{ \prn{\frac{\Lf R}{\epsilon}}^{2/3} \wedge \prn{\frac{\Lg R^2}{\epsilon}}^{1/3}}}$ lower bound follows from our construction in the previous section. Fix any $T>1$ and $\nsm\ge 0$ and let $d = \ceil*{T+ 4\alpha_T^{-2}\log 4NT}$ with $\alpha_T = 1/(4T^{3/2})$.
	Let $\perm$ be random permutation of $[N]$ and let $U$ be drawn uniformly from the set of  $d\times T$ orthogonal matrices. For $i\in[N]$, let
	\begin{equation*}
		\tf_i(x) = \fhard_{\iperm(i)}(U^\top x).
	\end{equation*}
	Then, by \Cref{prop:prog-control} and \Cref{lem:hard-instance-props}.\ref{item:hi-zc}, any optimization algorithm interacting with $(\tf_i)_{i\le N}$ satisfies with probability as least $1/2$ that $\prog[\alpha_T](U^\top x_i) < T$ for every $i \le \frac{1}{64}NT$. 

	Set
	\begin{equation*}
		T = \floor*{\frac{1}{5}\brk*{ \prn[\Big]{\frac{\Lf R}{\epsilon}}^{2/3} \wedge \prn[\Big]{\frac{\Lg R^2}{\epsilon}}^{1/3}}}
		~~\mbox{and}~~
		\nsm = \frac{\Lg R}{\Lf}.
	\end{equation*}
	We may assume $T \le N/2$ without loss of generality, since otherwise the second term in the lower bound dominates. Let
	\begin{equation*}
		f_i(x) = \Lf R \tf_i(x/R).
	\end{equation*} 
	With these settings, we have that $f_i$ is  both $\Lf$-Lipschitz and $\Lg$-smooth for every $i$ due to \Cref{lem:hard-instance-props}.\ref{item:hi-lip}, the choice of $\nsm$, and the fact that $U$ is orthogonal. Moreover, if $x_1, x_2, \ldots$ are the iterates of an algorithm interacting with a finite sum oracle for $\problemInstance$ then $x_1/R, x_2/R, \ldots$ are the iterates of an algorithm interacting with $(\tf_i)_{i\le[N]}$. Therefore, by the above discussion, with probability at least $1/2$ the first
	\begin{equation*}
		\frac{1}{64}NT = \Omega \prn*{  
			N \brk*{ \prn[\Big]{\frac{\Lf R}{\epsilon}}^{2/3} \wedge \prn[\Big]{\frac{\Lg R^2}{\epsilon}}^{1/3}}}
	\end{equation*}
	iterates of the algorithm satisfy $\prog[\alpha_T](U^\top x) < T$, and  consequently, by \Cref{lem:hard-instance-props}.\ref{item:hi-optgap}, they are 
	\begin{equation*}
		\Lf R \min\crl*{ \frac{1}{8T^{3/2}}, \frac{\nsm}{32T^3} }
		=
		\min\crl*{ \frac{\Lf R}{8T^{3/2}}, \frac{\Lg R^2}{32T^3} } > \epsilon
	\end{equation*}
	suboptimal for $\max_{i\in[N]}f_i$. To conclude the linear in $N$ lower bound, we note that since the sub-optimality bound holds with probability at least $1/2$ over $\perm$, $U$ and the algorithm randomness,  for every algorithm there must exist fixed $\perm$ and $U$ for which the bound holds with probability $1/2$ over the randomness of the algorithm alone.
	
	To show the remaining $\Omega \prn[\big]{ \brk[\big]{ \prn[\big]{\frac{\Lf R}{\epsilon}}^2 \wedge \prn[\big]{\frac{N \Lg R^2}{\epsilon}}^{1/2}}}$ term in the lower bound, we
	recall the following classical result. 
	For every $R'>0$ there exists a distribution over functions $F:\R_d \to \R$ with $d'=O \prn*{ \brk*{
			\prn[\big]{\frac{\Lf {R'}}{\epsilon}}^{2} \wedge \prn[\big]{\frac{\Lg {R'}^2}{\epsilon}}^{1/2}}
		\log \frac{N\Lf {R'}}{\epsilon } }$ that are $\Lg$-Lipschitz, $\Lf$-smooth and has a global minimizer with norm at most $R$, such that for any algorithm interacting with $F$, with probability at least $1-\frac{1}{2N}$, the first
	\begin{equation*}
		T_{R'} = \Omega\prn*{ \brk*{ \prn[\Big]{\frac{\Lf {R'}}{\epsilon}}^2 \wedge \prn[\Big]{\frac{Lg {R'}^2}{\epsilon}}^{1/2}}}
	\end{equation*}
	iterations are $\epsilon$ suboptimal for $F$. This result follows from smoothing Nemirovski's function; see~\citet[][Theorem 14 with $p=2$ and $\kappa=0,1$]{diakonikolas2020lower}. 
	To  strengthen the smooth term in the lower bound, we consider $N$ copies of this construction with $R' = R/\sqrt{N}$ that operate on distinct coordinates, i.e., we let $f_i(x)=F(x\coind{1+(i-1)d'}, \ldots, x\coind{id'})$, so that the global minimizer of $\max_i f_i(x)$, which consists of $N$ copies of the global minimizer of $F$, has norm at most $R'\sqrt{N}=R$, and consequently we may constrain the domain to $\ball_{N}^{d'N}$ without decreasing the optimality gap of any point in the ball. For any algorithm interacting with $\problemInstance$, the first 
	\begin{equation*}
		N T_{R/\sqrt{N}}
		=
		\Omega \prn*{ \brk*{ \prn[\Big]{\frac{\Lf {R}}{\epsilon}}^2 \wedge \prn[\Big]{\frac{N Lg {R}^2}{\epsilon}}^{1/2}}}
	\end{equation*}
	queries of the algorithm must select one of the $N$ components at most $T_{R/\sqrt{N}}$ times and therefore (with probability at least $\half$) be $\epsilon$ suboptimal for at least one component, and hence for their maximum. This gives the sublinear in $N$ term of the lower bound. We remark that the ``hard instance duplication'' argument is at the core of existing lower bounds for finite sum optimization~\cite{woodworth2016tight,fang2018near}; our argument for proving the linear in $N$ lower bound is inherently different.
\end{proof} %
\subsection{Extension to unconstrained setting}\label{app:lb-unconstrained}

\newcommand{\mfhard}{\bar{f}^{\{N,T,\nsm\}}}
\newcommand{\mFhard}{\bar{F}_{\max}^{\{N,T,\nsm\}}}
While we state and prove our lower bound for optimization problems whose domain is a ball of radius $R$, it also extends to the case of unconstrained setting of our upper bounds. That is, when the domain is $\R^d$ and we are guaranteed a local minimizer exists in a radius of $R$ from the initial point. One way to show this extension is the technique of~\cite{diakonikolas2020lower} where we replace $\fhard_i$ with 
\begin{equation}\label{eq:hard-instance-with-norm-barrier}
	\mfhard_i(x) \defeq \min_{y\in \R^d} \crl*{\max\crl*{\fhard_i(y), \norm{y}-1-\nsm^{-1}}  + \frac{\ell}{2}\norm{y-x}^2}.
\end{equation}

The definition above consists of two modification: pairwise maximum with $\norm{\cdot}-1$ and infimal convolution with $\frac{\nsm}{2}\norm{\cdot}^2$. The pairwise maximum  guarantees that  large norm queries cannot break the zero-chain progress control mechanism: responses to query points with norm  $\Omega(1)$ will not depend on the random transformation $U$ at all, while for responses with norm $O(1)$ we can control the progress using the zero-chain structure of $\fhard$ as before.  The infimal convolution guarantees the function remains $\ell$ smooth (and also does not increase the Lipschitz constant). 

Finally, it remains to check that the new construction still satisfies property~\ref{item:hi-optgap} of \Cref{lem:hard-instance-props} up to a constant. To see that it does, first note that the global minimizer of $\mFhard(x)\defeq \max_{i\in[N]} \mfhard_i(x)$ still satisfies ${\xopt}\coind{i}= 1/\sqrt{T}$ for all $i\le T$, and that $\mFhard(\xopt)=0$. 
Moreover, we clearly have
\begin{equation*}
		\mfhard_i(x) \ge  \min_{y\in \R^d} \crl*{\linkfun[\alpha_T,\nsm](y/2)  + \frac{\ell}{2}\norm[\Big]{y-\frac{x\coind{i}-x\coind{i-1}}{2}}^2} \defeq \tilde{\psi}\prn*{\frac{x\coind{i}-x\coind{i-1}}{2}},
\end{equation*}
and it is not hard to verify that $\tilde{\psi}(t) \ge \linkfun[\alpha_T, \nsm](c t)$ for some constant $c>0$ (in fact, $\tilde{\psi}(t) = \linkfun[\alpha_T, c\nsm](t)$).

\end{document}